\documentclass[final, twoside]{amsart}
\usepackage[latin1]{inputenc}
\usepackage[T1]{fontenc}
\usepackage{textcomp}
\usepackage[english]{babel}
\usepackage{amsmath}
\usepackage{amsfonts}
\usepackage{amssymb}
\usepackage{amsopn}
\usepackage{amscd}
\usepackage{amsthm}
\usepackage[sc]{mathpazo}
\usepackage{mathrsfs}
\usepackage{graphicx}
\usepackage{enumerate}

 \usepackage{hyperref}
 \hypersetup{
   colorlinks=true,
   linkcolor=blue,
   citecolor=blue,
   linktoc=page,
   pdfauthor={Yohann de Castro}
 }
\newcommand{\Eplus}[1]{E^+_{#1}}
\newcommand{\Eminus}[1]{E^-_{#1}}
\renewcommand{\S}{\mathcal{S}}
\newcommand{\J}{\mathcal{J}}
\newcommand{\Che}{\mathfrak{T}}
\newcommand{\T}{\mathcal{T}}
\newcommand{\isoT}{\varTheta_\T}
\renewcommand{\d}{\mathrm{d}}
\newcommand{\sgn}{\mathrm{sgn}}
\newcommand{\N}{\mathbb{N}}                
\newcommand{\R}{\mathbb{R}}                     
\newcommand{\M}{\mathcal{M}}
\newcommand{\K}{\mathcal{K}}
\newcommand{\sigmastar}{\sigma^\star}
\newcommand{\normTV}[1]{\left\lVert #1\right\lVert_{TV}}
\newcommand{\abs}[1]{\left\vert#1\right\vert}
\newcommand{\norm}[1]{\left\lVert#1\right\lVert}
\newtheoremstyle{theo}
{}
{}
{\itshape}
{\parindent}
{\bf}
{\ ---}
{.5em}
{}%
\theoremstyle{theo}
\newtheorem{theorem}{Theorem}[section]
\newtheorem{lemma}[theorem]{Lemma}
\newtheorem{proposition}[theorem]{Proposition}
\newtheorem*{cor}{Corollary}

\newtheoremstyle{def}%
{}
{}
{\itshape}
{\parindent}
{\bf}
{\ ---}
{.5em}
{}%
\theoremstyle{def}
\newtheorem{definition}{Definition}[section]

\theoremstyle{remark}
\newtheorem*{remark}{Remark}

\DeclareMathOperator{\argmin}{Arg\,min}
\DeclareMathOperator{\Index}{Index}

\begin{document}
\pagestyle{headings} 
\title[Exact Reconstruction using GME]{Exact Reconstruction using
Beurling Minimal Extrapolation}
\author{Yohann de Castro}
\author{Fabrice Gamboa}
\address{Institut de Math\'ematiques de Toulouse (CNRS UMR 5219). Universit\'e Paul Sabatier,
118 route de Narbonne, 31062 Toulouse, France.}
\email{yohann.decastro@math.univ-toulouse.fr}
\email{fabrice.gamboa@math.univ-toulouse.fr}

\begin{abstract}
We show that measures with finite support on the real line are the
unique solution to an algorithm, named generalized minimal extrapolation, involving only a finite
number of
generalized moments (which encompass the standard moments, the
Laplace transform, the Stieltjes transformation, etc).

Generalized minimal extrapolation shares related geometric properties with basis pursuit of Chen,
Donoho and
Saunders \cite{MR1639094}. Indeed we also extend some standard results of compressed
sensing (the dual polynomial, the nullspace property) to the signed measure framework.

We express exact reconstruction in terms of a simple interpolation problem.
We prove that every nonnegative measure, supported by a set containing $s$ points, can
be exactly recovered from only $2s+1$ generalized moments. This result leads to a new construction
of deterministic sensing matrices for compressed sensing. 
\end{abstract}
\maketitle

\section*{Introduction}
In the last decade much emphasis has been put on the exact reconstruction of sparse finite
dimensional vectors using the \textit{basis pursuit} algorithm. The pioneering paper of Chen,
Donoho and Saunders \cite{MR1854649} has brought this method to the statistics community.
Note that the seminal ideas on the subject appeared in earlier works of Donoho and Stark
\cite{MR997928}. Therein, mainly the discrete Fourier transform is considered. Similarly,
P. Doukhan, E. Gassiat and one author of this present paper \cite{MR1393035,MR1393430}
considered the exact reconstruction of a nonnegative measure. More precisely, they derived results
when one only knows the values of a finite number of linear functionals at the target measure.
Moreover, they study stability with respect to a metric for weak convergence which is not the case
here.

In this paper, we are concerned with the measure framework. We show that the exact reconstruction
of a \textbf{signed} measure is still possible when one only knows a finite number of non-adaptive
linear measurements. Surprisingly our method, called \textit{generalized minimal extrapolation},
appears to uncover exact reconstruction results related to basis pursuit.

Let us explain more precisely what is done here. Consider a \textit{signed
discrete measure} $\sigma$ on a set $I$. \textit{Unless otherwise specified}, assume that
$I:=[-1,1]$.
Note that all our results easily extend to any real bounded set. Consider the \textit{Jordan
decomposition},
\[\sigma=\sigma^+-\sigma^-,\]
and denote by $\S^+$ (resp.\ $\S^-$) the support of $\sigma^+$ (resp.\ $\sigma^-$).
Let us define the \textit{Jordan support} of the measure $\sigma$ as the pair
$\mathcal J:=(\S^+,\S^-)$. Assume further that $\S:=\S^+\cup\S^-$ is \textbf{finite} and has
cardinality $s$. Moreover suppose that
$\J$ belongs to a family $\varUpsilon$ of pairs of subsets of $I$ (see Definition \ref{DEF
TYPE} for more details). We call $\varUpsilon$ a \textit{Jordan support family}. The measure
$\sigma$ can be written as
\begin{equation}\notag
 \sigma=\sum_{i=1}^s\,\sigma_i\,\delta_{x_i}\,,
\end{equation}
where $\S=\{x_1,\dotsc,x_s\}$, $\sigma_1,\dotsc,\sigma_s$ are nonzero real numbers, and
$\delta_{x}$ denotes the Dirac
measure at point $x$.

Let $\mathcal F=\{u_0,u_1,\dotsc,u_n\}$ be \textbf{any} family of
continuous functions on $\overline I$, where the set $\overline I$ denotes the closure of
$I$ (this statement is meant to be general and encompasses the case where $I$ is not closed). Let
$\mu$ be a signed measure on $I$. The $k$-th \textit{generalized moment} of $\mu$ is defined by
\begin{equation}\label{generalized moments}
 c_{k}(\mu)=\displaystyle\int\nolimits_I u_{k}\,\d\mu
\end{equation}
for all the indices $k=0,1,\dotsc,n$.
\subsection*{Our main issue} We are concerned with the reconstruction of the \textit{target
measure}  $\sigma$ from the \textit{observation} of $\mathcal
K_n:=(c_0(\sigma),\dotsc,c_{n}(\sigma))$, i.e. its first $(n+1)$ generalized moments. We assume
that both the support $\mathcal S$ and the weights $\sigma_i$ of the target measure $\sigma$ are
\textbf{unknown}. We investigate if it
is possible to recover $\sigma$ uniquely from the observation of $\K_n$. More precisely,
\textit{does an algorithm fitting $\K_n(\sigma)$ among \textbf{all} the signed measures of $I$
recover the measure $\sigma$?}

Note that a finite number of assigned standard moments does not define a unique signed measure.
In fact one can check that for each signed measure $\mu$ and for each integer $m\geq1$ there
exists a measure $\mu'\neq\mu$ having the same first $m$ moments. It seems there is no hope of
recovering
discrete measures from a finite number of its generalized moments. Surprisingly, we show that
every \textit{extrema Jordan type} measure $\sigma$ (see Definition \ref{DEF TYPE}
and the examples that follow) is the unique solution of a \textit{total variation} minimizing
algorithm, \textit{generalized minimal extrapolation}. 

\subsection*{Basis pursuit}
In \cite{MR1639094} Chen, Donoho and Saunders introduced \textit{basis
pursuit}. It is the process of reconstructing a target vector $x_0\in\R^p$ from the observation
$b=Ax_0$ by finding a sparse solution $x^\star$ to an under-determined system of equations:
\begin{equation}\label{BasisPursuit}\tag{$\mathrm{BP}$}
  x^\star\in\argmin\displaylimits_{y\in\R^p}\norm{y}_1\quad \text{s.t.} \ Ay=Ax_0\,,
\end{equation}
where $A\in\R^{n\times p}$ is the \textit{design matrix}. This program is one of the other first
steps \cite{MR2236170,MR2241189} of a remarkable
theory so-called \textit{compressed sensing}. As a result, this extremum is appropriated to the
reconstruction of \textit{sparse} vectors (i.e.\ vectors with a small support\ \cite{MR2241189}).
In this paper we develop a related program that recovers all the measures
with enough structured Jordan support (which can be seen as the sparsity-related measures).

\subsection*{Generalized minimal extrapolation}
Denote by $\M$ the set of finite signed measures on $I$ and by $\normTV{\,.\,}$ the
\textit{total variation} norm. We recall that for all $\mu\in\M$,
\[\normTV{\mu}=\sup_{\Pi}\sum_{E\in\Pi}\abs{\mu(E)}\,,\]
where the supremum is taken over all partitions $\Pi$ of $I$ into a finite number of
disjoint
measurable subsets. By analogy with basis pursuit, \textit{generalized minimal extrapolation} is
the process of reconstructing a target measure $\sigma$ from the observation
$\K_n(\sigma)=(c_0(\sigma),\dotsc,c_{n}(\sigma))$ of its first $n+1$ generalized moments
$c_k(\sigma)$ by finding a solution of the problem
\begin{equation}\label{support pursuit}\tag{$\mathrm{GME}$}
 \sigmastar\in\argmin\displaylimits_{\mu\in\M}\normTV{\mu}\quad \text{s.t.} \
\K_n(\mu)=\K_n(\sigma)\,.
\end{equation}

\noindent On one hand, basis pursuit minimizes the $\ell_1$-norm subject to linear
constraints. On the other hand, generalized minimal extrapolation naturally substitutes the
$TV$-norm (the total variation norm) for the $\ell_1$-norm. For the case of Fourier coefficients,
\eqref{support pursuit} is simply {\it Beurling Minimal Extrapolation}
\cite{beurling1938integrales}. The program \eqref{support pursuit} is named after this remark.

Let us emphasize that generalized minimal extrapolation looks for a minimizer among \textbf{all}
\textbf{signed} measures on $I$. Nevertheless, the target measure $\sigma$ is assumed to be of
\textit{extrema Jordan type}.

\subsubsection*{Extrema Jordan type measures} Let us define more precisely what we understand by
the Jordan support family $\varUpsilon$.
\begin{definition}[Extrema Jordan type measure]\label{DEF TYPE}
We say that a signed measure $\mu$ is of extrema Jordan type $($with respect to a family
$\mathcal F=\{u_0,u_1,\dotsc,u_n\})$ \textbf{if and only if} its
\textit{Jordan decomposition} $\mu=\mu^+-\mu^-$ satisfies
\begin{equation*}
 \mathrm{Supp}(\mu^+)\subset\Eplus P\quad\mathrm{and}\quad
\mathrm{Supp}(\mu^-)\subset\Eminus P,
\end{equation*}
where $\mathrm{Supp}(\nu)$ is defined as the support of the measure $\nu$, and
\begin{itemize}
 \item $P$ denotes any linear combination of elements of $\mathcal F$,
\item $P$ is not constant and $\norm P_\infty\leq1$,
\item $\Eplus P$ $($resp.\ $\Eminus P)$ is the set of all points $x_i$ such that $P(x_i)=1$
$($resp. $P(x_i)=-1)$.
\end{itemize}
\end{definition}
\noindent In the following, we give some examples of extrema Jordan type measures with respect to
the family
\[\mathcal F_p^n=\{1,x,x^2,\dotsc,x^n\}\,.\]
These measures can be seen as "interesting" target measures for \eqref{support pursuit} given
observation of the first $n+1$ standard moments.

\subsubsection*{Examples with respect to the family $\mathcal F_p^n$}
For the sake of readability, let $n=2m$ be an even integer. We present three important
examples.
\begin{description}
 \item[Nonnegative measures]
The nonnegative measures whose support has size $s$ not greater than $n/2$ are extrema Jordan
type measures. Indeed, let $\sigma$ be a nonnegative
measure and $\S=\{x_1,\dotsc,x_s\}$ be its support. Set \[P=1-c\prod_{i=1}^s(x-x_i)^2\,.\]
Then, for a sufficiently small value of the parameter $c$, the polynomial $P$ has supremum norm
not greater than $1$. The existence of such a polynomial shows that the measure $\sigma$ is an
extrema Jordan type measure.

\noindent In Section \ref{SRPM} we extend this notion to any \textit{homogeneous $M$-system}.
\item[Chebyshev measures]
The $k$-th \textit{Chebyshev polynomial of the first order} is defined by
\begin{equation}\label{Chebyshev}
 T_k(x)=\cos(k\arccos(x)),\quad\forall x\in[-1,1]\,.
\end{equation}
It is well known that it has supremum norm not greater than $1$, and that
\begin{itemize}
 \item  $\Eplus{T_k}=\big\{\cos(2l\pi/k),\
l=0,\dotsc,\big\lfloor\frac
k2\big\rfloor\big\}$,
  \item $\Eminus{T_k}=\big\{\cos((2l+1)\pi/k),\
l=0,\dotsc,\big\lfloor\frac k2\big\rfloor\big\}$,
\end{itemize}
whenever $k>0$. Then, any measure $\sigma$ such that
\begin{equation*}
 \mathrm{Supp}(\sigma^+)\subset\Eplus {T_k}\quad\mathrm{and}\quad
\mathrm{Supp}(\sigma^-)\subset\Eminus {T_k},
\end{equation*}
for some $0<k\leq n$, is an extrema Jordan type measure.

\noindent Further examples are presented in Section \ref{Tigre}.
\item[$\Delta$-spaced out type measures]
Let $\Delta$ be a positive real and $S_\Delta$ be the set of all pairs $(S^+,S^-)$ of subsets of
$[-1,1]$
such that \[\forall x,y\in S^+\cup S^-,\ x\neq y,\quad \abs{x-y}\geq\Delta.\]
In Lemma \ref{Delta}, we prove that, for all $(S^+,S^-)\in S_\Delta$, there exists a
polynomial $P_{(S^+,S^-)}$ such that
\begin{itemize}
  \item $P_{(S^+,S^-)}$ has degree $n$ not greater than a bound depending only on $\Delta$,
 \item $P_{(S^+,S^-)}$ is equal to $1$ on the set $S^+$,
\item $P_{(S^+,S^-)}$ is equal to $-1$ on the set $S^-$,
\item and $\lVert {P_{(S^+,S^-)}}\lVert_\infty\leq1$.
\end{itemize}
This shows that any
measure $\sigma$ with Jordan support included in $S_\Delta$ is an extrema Jordan type measure.
\end{description}
In this paper, we give exact reconstruction results for these three kinds of extrema Jordan type
measures. In fact, our results extend to others families $\mathcal F$. Roughly, they can
be stated as follows:
\begin{description}
 \item [Nonnegative measures] Assume that $\mathcal F$ is a homogeneous $M$-system (see
\ref{homogeneous}). Theorem \ref{Exact reconstruction Theorem} shows that \textit{any
nonnegative measure $\sigma$ is the \textbf{unique} solution of generalized minimal extrapolation
given the
observation $\K_n(\sigma)$, where $n$ is not less than twice the size of the support of $\sigma$.}
  \item [Generalized Chebyshev measures]
Assume that $\mathcal F$ is an $M$-system (see definition \ref{Msystems}). Proposition
\ref{Cheby Exact} shows the following result:
\textit{Let $\sigma$ be a signed measure having Jordan support included in
$(\Eplus{\Che_k},\Eminus{\Che_k})$, for some $1\leq k\leq n$,
where $\Che_k$ denotes the $k$-th generalized Chebyshev polynomial $($see \ref{Def Cheby}$)$.
\textbf{Then} $\sigma$ is the \textbf{unique} solution to generalized minimal extrapolation
\eqref{support pursuit}
given $\K_n(\sigma)$, i.e.\ its first $(n+1)$ generalized moments.}
  \item[$\Delta$-interpolation] Considering the standard family $\mathcal
F_p^n=\{1,x,x^2,\dotsc,x^n\}$, Proposition \ref{NSP Delta} shows that \textit{generalized minimal
extrapolation
exactly recovers any $\Delta$-spaced out type measure $\sigma$ from the observation
$\K_n(\sigma)$, where $n$ is greater than a bound depending only on  $\Delta$.}
\end{description}
These results are closely related to standard results of basis pursuit \cite{MR2241189}. In fact,
further analogies with compressed sensing can be emphasized.

\subsubsection*{Analogy with compressed sensing}\label{CS}
Our estimator follows the aura of the recent breakthroughs \cite{MR1639094, MR2236170} in
compressed sensing.

In the past decade E. J. Cand\`es, J. Romberg, and T. Tao have shown
\cite{MR2230846} that it is possible to exactly recover all {sparse} vectors from few linear
measurements. They considered a matrix $A\in\R^{n\times p}$ with i.i.d entries (centered
Gaussian, Bernoulli, random Fourier sampling) and an $s$-sparse vector $x_0$ (i.e.\ vector with
support of size at most $s$). They pointed
out that, with very high probability, the vector $x_0$ is the only point of contact between the
$\ell_1$-ball of radius $\norm{x_0}_1$ and the affine space $\{y,\ Ay=Ax_0\}$. This result holds
as soon as $n\geq C\,s\log(p/s)$, where $C>0$ is a universal constant . In our framework we
uncover the same geometric property:
\medskip

\noindent\textit{Let $\sigma$ be an extrema Jordan type measure. Then $\sigma$ is a point
of contact between the ball of radius $\Vert\sigma\Vert_{TV}$ and the affine space
$\{\mu\in\mathcal M,\ \K_n(\mu)=\K_n(\sigma)\}$, where $n$ is greater than a bound depending only
on the structure of the Jordan support of $\sigma$. For instance, in the nonnegative measure case,
if $\sigma$ has support of size
at most $s$, then $n=2s$ suffices $($see Theorem \ref{Exact reconstruction Theorem}$)$.}
\medskip

\noindent Actually the reader can check that the above property is equivalent to the fact that the
measure
$\sigma$ is a solution of generalized minimal extrapolation (more details can be found in Section
\ref{Cone}).
Accordingly, generalized minimal extrapolation \eqref{support pursuit} minimizes the total
variation in order to
pursue support of the target measure.

\subsection*{Organization}
This paper falls into four parts. The next section introduces \textit{generalized dual
polynomials} and shows that exact recovery can be understood in terms of an
interpolation problem. Section 2 studies the exact reconstruction of nonnegative
measures, and gives \textit{explicit} construction of design matrices for basis pursuit.
Section 3 focuses on generalized Chebyshev polynomials and shows that it is possible
to reconstruct signed measures from very few generalized moments. The last section uncovers a
property related to the nullspace property of compressed sensing.
\section{Generalized dual polynomials}\label{GDP}
In this section we introduce generalized dual polynomial. In particular we are concerned with a
sufficient condition that guarantees the exact reconstruction of the measure $\sigma$.
In fact, this condition relies on an interpolation problem.
\subsection{An interpolation problem}
An insight into exact reconstruction is given
by Lemma \ref{dual polynomial lemma}. Roughly, the existence of a
generalized dual polynomial is a sufficient condition for the exact reconstruction of a signed
measure with finite support.

As usual, the following result holds for any family $\mathcal F=\{u_0,u_1,\dotsc,u_n\}$ of
continuous functions on $\overline{I}$. Throughout, $\sgn(x)$ denotes the
sign of the real $x$. 
\begin{lemma}[The generalized dual polynomials]\label{dual polynomial lemma}
Let $n$ be a positive integer. Let $\S=\{x_1,\dotsc,x_s\}\subset I$ be a subset of size $s$
and $(\varepsilon_1,\dotsc,\varepsilon_s)\in\{\pm 1\}^s$. \textbf{If} there exists a linear
combination $P=\sum_{k=0}^n a_k u_k$ such that
\begin{enumerate}[$(i)$]
  \item the generalized Vandermonde system
\[\begin{pmatrix}
    u_0(x_1) & u_0(x_2) & \dotsc & u_0(x_s)\\
    u_1(x_1) & u_1(x_2) & \dotsc & u_1(x_s)\\
    \vdots & \vdots &  & \vdots\\
    u_{n}(x_1) & u_{n}(x_2) & \dotsc & u_{n}(x_s)
  \end{pmatrix}\]
has full column rank,
  \item $P(x_i)=\varepsilon_i,\ \forall\,i=1,\dotsc,s$,
  \item $\abs{P(x)}<1,\ \forall x\in[-1,1]\setminus\S$,
\end{enumerate}
\textbf{Then} every measure $\sigma=\sum_{i=1}^s\,\sigma_i\,\delta_{x_i}$, such
that $\sgn(\sigma_i)=\varepsilon_i$, is the \textbf{unique} solution of generalized minimal
extrapolation given the observation $\K_{n}(\sigma)$.
\end{lemma}
\begin{proof}
See \ref{proof dual polynomial proposition}.
\end{proof}
\noindent The linear combination $P$ considered in the Lemma \ref{dual
polynomial lemma} is called a \textbf{generalized dual polynomial}. This naming is inherited
from the original article
\cite{MR2236170} of Cand\`es, Tao and Romberg, and the \textit{dual certificate} named
by Cand\`es and Plan \cite{RIPless}.

\subsection{Reconstruction of a cone}\label{Cone}
Given a subset $\S=\{x_1,\dotsc,x_s\}$ and a sign sequence
$(\varepsilon_1,\dotsc,\varepsilon_s)\in\{\pm 1\}^s$, Lemma \ref{dual polynomial lemma} shows that
\textit{if} the generalized interpolation problem defined by $(i)$, $(ii)$ and $(iii)$ has a
solution \textit{then} generalized minimal extrapolation recovers exactly all measures $\sigma$
with support $\S$ and
such
that $\sgn(\sigma_i)=\varepsilon_i$.

Let us emphasize that the result is slightly stronger. Indeed the proof of \ref{proof dual
polynomial
proposition} remains unchanged if some coefficients $\sigma_i$ are zero. Consequently
\eqref{support pursuit} recovers exactly all the measures $\sigma$ of which support is
\textbf{included} in
$\S=\{x_1,\dotsc,x_s\}$ and such that
$\sgn(\sigma_i)=\varepsilon_i$ for all nonzero $\sigma_i$.

Let us denote this set by $\mathcal
C(x_1,\varepsilon_1,\dotsc,x_s,\varepsilon_s)$. It is exactly the \textit{cone} defined by
\[\mathcal C(x_1,\varepsilon_1,\dotsc,x_s,\varepsilon_s)
=\Big\{\sum_{i=1}^s\,\mu_i\,\delta_{x_i}\ \Big|\
\forall\mu_i\neq0,\ \sgn(\mu_i)=\varepsilon_i\Big\}.\]
Thus \textit{the existence of $P$ implies the exact reconstruction of \textbf{all}
measures in this cone}. The cone $\mathcal C(x_1,\varepsilon_1,\dotsc,x_s,\varepsilon_s)$ is the
conic span of an
$\textstyle(s-1)$-dimensional face of the $TV$-unit ball, that is
\[\mathcal F(x_1,\varepsilon_1,\dotsc,x_s,\varepsilon_s)
=\Big\{\sum_{i=1}^s\,\varepsilon_i\lambda_i\,\delta_{x_i}\ \Big|\
\forall\,i,\ \lambda_i\geq0\ \mathrm{and}\ \sum_{i=1}^s\lambda_i=1\Big\}.\]
Furthermore, the affine space $\{\mu,\
\K_n(\mu)=\K_n(\sigma)\}$ is tangent to the $TV$-unit ball at any point $\sigma\in\mathcal
F(x_1,\varepsilon_1,\dotsc,x_s,\varepsilon_s)$, as shown in the following remark.
\begin{remark}\label{subgradient}
From a convex optimization point of view, the \textit{dual certificates} \cite{RIPless} and the
generalized dual polynomials are deeply related: the existence of a
generalized dual polynomial $P$
implies that, for all $\sigma\in\mathcal
F(x_1,\varepsilon_1,\dotsc,x_s,\varepsilon_s)$, a subgradient $\Phi_P$ of the $TV$-norm
at the point $\sigma$ is perpendicular to the set of the feasible points, that is \[\{\mu,\
\K_n(\mu)=\K_n(\sigma)\}\subset\ker(\Phi_P),\]
where $\ker$ denotes the nullspace. A proof of this remark can be found in \ref{proof
subgradient}.
\end{remark}

\subsection{On condition (i) in Lemma \ref{dual polynomial lemma}}
Obviously, when $u_k=x^k$ for $k=0,1,\dotsc,n$, conditions $(ii)$ and $(iii)$ imply that $n\geq
s$ and so condition $(i)$. Nevertheless, this implication is not
true for a general set of functions $\{u_0,u_1,\dotsc,u_n\}$. Moreover, Lemma \ref{dual polynomial
lemma} can fail if condition $(i)$ is not satisfied. For example, set $n=0$ and consider a
continuous function $u_0$ satisfying the two conditions $(ii)$ and $(iii)$. In this case,
\textit{if} the target $\sigma$ belongs to $\mathcal
F(x_1,\varepsilon_1,\dotsc,x_s,\varepsilon_s)$ (where $x_1,\dotsc,x_s$ and
$\varepsilon_1,\dotsc,\varepsilon_s$ are given by $(ii)$ and $(iii)$), \textit{then} \textbf{every}
measure $\mu\in\mathcal F\big(x_1,\varepsilon_1,\dotsc,x_s,\varepsilon_s\big)$ is a solution of
generalized minimal extrapolation given the observation $\K_0(\sigma)$. Indeed,
\[\normTV\mu=\int_{-1}^1u_0\,\d\mu=\K_0(\mu),\]
for all $\mu\in\mathcal F\big(x_1,\varepsilon_1,\dotsc,x_s,\varepsilon_s\big)$. This example shows
that condition $(i)$ is necessary. Reading the proof \ref{proof dual polynomial proposition},
conditions $(ii)$ and $(iii)$ ensure that the solutions to generalized minimal extrapolation belong
to the
cone $\mathcal C(x_1,\varepsilon_1,\dotsc,x_s,\varepsilon_s)$, whereas condition $(i)$
gives uniqueness.

\subsection{The extrema Jordan type measures}

Lemma \ref{dual polynomial lemma} shows that Definition \ref{DEF TYPE} is well-founded. In fact,
we have the the following corollary.
\begin{cor}
 Let $\sigma$ be an extrema Jordan type measure. \textbf{Then} the measure $\sigma$ is a solution
to generalized minimal extrapolation given the observation $\K_{n}(\sigma)$.

Furthermore, \textbf{if} the Vandermonde system given by $(i)$ in Lemma \ref{dual polynomial
lemma} has full column rank $($where $\S=\{x_1,\dotsc,x_s\}$ denotes the support of $\sigma)$,
\textbf{then} the measure $\sigma$ is the \textbf{unique} solution to generalized minimal
extrapolation given the
observation $\K_{n}(\sigma)$.
\end{cor}
\noindent This corollary shows that the "extrema Jordan type" notion is appropriate to exact
reconstruction using generalized minimal extrapolation.
%
\section{Exact reconstruction of the nonnegative measures}\label{SRPM}
In this section we show that \textbf{if} the underlying family
$\mathcal F=\{u_0,u_1,\dotsc,u_n\}$ is a \textit{homogeneous $M$-system} \textbf{then}
\eqref{support pursuit} recovers exactly each finitely supported nonnegative measure $\mu$ from
the observation of a surprisingly few generalized moments. We begin with the definition of
homogeneous $M$-systems.

\subsection{\textit{Markov} systems}\label{Markov System}
\textit{Markov} systems were introduced in approximation theory
\cite{MR0458081,MR1367960,MR0204922}.
They deal with the problem of finding the best approximation, in terms of the $\ell_\infty$-norm,
of a given continuous function in $\ell_\infty$ norm. We begin with the definition of
\textit{Chebyshev} systems
(the so-called $T$-system). They can be seen as a natural extension of algebraic monomials.
Thus a finite
combination of elements of a $T$-system is called a \textit{generalized polynomial}.
\subsubsection{$T$-systems of order $k$} Denote by $\{u_0,u_1,\dotsc, u_k\}$ a set of
continuous
real (or complex) functions on $\overline I$. This set is a $T$-system of degree $k$ \textit{if
and only if} every generalized  polynomial \[P=\sum_{l=0}^ka_l u_l\,,\] where
$(a_0,\dotsc,a_k)\neq(0,\dotsc,0)$, has at most $k$ zeros in $I$. 

\noindent This definition is equivalent to each of the two following conditions:
\begin{itemize}
  \item For all $x_0,x_1,\dotsc, x_k$ distinct elements of $I$ and all $y_0,y_1,\dotsc,y_k$
real (or complex) numbers, there exists a unique generalized polynomial $P$ (i.e.
$P\in\mathrm{Span} \{u_0,u_1,\dotsc, u_k\}$) such that $P(x_i) = y_i$, for all $i = 0,1,\dotsc,k$.
  \item For all $x_0,\dotsc, x_k$ distinct elements of $I$, \textit{generalized Vandermonde
system}
\[
\begin{pmatrix}
    u_0(x_0) & u_0(x_1) & \dotsc & u_0(x_k)\\
    u_1(x_0) & u_1(x_1) & \dotsc & u_1(x_k)\\
    \vdots & \vdots &  & \vdots\\
    u_{k}(x_0) & u_{k}(x_1) & \dotsc & u_{k}(x_k)
  \end{pmatrix}
\]
has full rank.
\end{itemize}

\subsubsection{$M$-systems}\label{Msystems}
We say that the family $\mathcal F=\{u_0,u_1,\dotsc,u_n\}$ is an $M$-system \textit{if and only if}
it is a $T$-system of degree $k$ for all $0\leq k\leq n$. Actually, $M$-systems are
common objects (see \cite{MR0458081}). We mention some examples below.

In this paper, we are concerned with target measures on $I=[-1,1]$. Usually $M$-systems are
defined
on general Hausdorff spaces (see \cite{MR1299454} for instance). For the sake of
readability, we present examples with different values of $I$. In each case, our results easily
extend to target measures with finite support included in the corresponding $I$. As usual, if not
specified, the set $I$ is assumed to be $[-1,1]$.

\begin{description}
 \item[Real polynomials] The family $\mathcal F_p=\{1,x,x^ 2,\dotsc\}$ is an
$M$-system. The real polynomials give the standard moments.
 \item[M\"untz polynomials] Let $0<\alpha_1<\alpha_2<\dotsb$ be any real numbers.
The family $\mathcal F_m=\{1,x^ {\alpha_1},x^ {\alpha_2},\dotsc\}$ is an $M$-system on
$I=[0,+\infty)$.
  \item[Trigonometric functions] The family $\mathcal
F_{\cos{}}=\left\{1,\cos(\pi x),\cos(2\pi x),\dotsc\right\}$ is an
$M$-system on $I=[0,1]$.
 \item[Characteristic function] The family $\mathcal
F_c=\left\{1,\exp(\imath\pi x),\exp( \imath 2\pi x),\dotsc\right\}$ is an
$M$-system on $I=[-1,1)$. The moments are the \textit{characteristic function} of
$\sigma$ at points $k\pi$, $k\in\N$. It yields
\[c_k(\sigma)=\displaystyle\int\nolimits_{-1}^1\exp(\imath k\pi
t)\d\sigma(t)=\varphi_{\sigma}(k\pi)\,.\]
In this case, the underlying scalar field is $\mathbb C$.
\item[Stieltjes transformation] The
family $\mathcal
F_s=\big\{\frac1{z_1-x},\frac1{z_2-x},\dotsc\big\}$, where none of the $z_k$'s belongs
to $[-1,1]$, is an $M$-system. The corresponding moments are the
\textit{Stieltjes
transformation}
$S_{\sigma}(z_k)$ of  $\sigma$, namely
\[c_k(\sigma)=\displaystyle\int\nolimits_{-1}^1\frac{\d\sigma(t)}{z_k-t}=S_{\sigma}(z_k)\,.\]
 \item[Laplace transform] The family $\mathcal
F_l=\left\{1,\exp(-x),\exp(-2x),\dotsc\right\}$ is an
$M$-system. The moments are the \textit{Laplace transform}
$\mathcal L\sigma$ at integer points, namely
\[c_k(\sigma)=\displaystyle\int\nolimits_{-1}^1\exp(-kt)\,\d\sigma(t)=\mathcal L\sigma(k)\,.\]
 \end{description}
A broad variety of common families can be considered in our framework. The above list is not meant
to be exhaustive.

Consider the family $\mathcal F_s=\big\{\frac1{z_0-x},\frac1{z_1-x},\dotsc\big\}$. Note that no
linear combination of its elements gives the constant function $1$. Thus the constant
function $1$ is not a generalized polynomial of this system. To treat such cases, we introduce
\textit{homogeneous} $M$-systems.
\subsubsection{Homogeneous $M$-systems}\label{homogeneous}
We say that a family $\mathcal F=\{u_0,u_1,\dotsc,u_n\}$ is a \textit{homogeneous} $M$-system
\textit{if and only if} it is an $M$-system and $u_0$ is a constant function. In this case, all
constant functions $c$, with $c\in\R$ (or $\mathbb C$), are generalized polynomials. Hence the
field
$\R$ (or $\mathbb C$) is naturally embedded in generalized polynomials. The adjective
homogeneous is named after this comment.

From any $M$-system we can always construct a homogeneous $M$-system.
Indeed, let $\mathcal F=\{u_0,u_1,\dotsc,u_n\}$ be an $M$-system. In particular the family
$\mathcal
F$
is a $T$-system of order $0$. Thus the continuous function $u_0$ does not vanish in $[-1,1]$. In
fact the family $\{1,\frac{u_1}{u_0},\frac{u_2}{u_0},\dotsc,\frac{u_n}{u_0}\}$ is a
homogeneous $M$-system.

All the previous examples of $M$-systems (see \ref{Msystems}) are homogeneous, even Stieltjes
transformation: \[\widetilde{\mathcal
F}_s=\Big\{1,\frac1{z_1-x},\frac1{z_2-x},\dotsc\Big\}\,.\]

\noindent Using homogeneous $M$-systems, we show that one
can exactly recover \textbf{all} nonnegative measures from a few generalized moments.

\subsection{An important theorem}
The following result is one of the main theorems of our paper. It states that the
generalized minimal extrapolation \eqref{support pursuit} recovers \textbf{all nonnegative
measures}
$\sigma$ whose support is of size $s$ from \textbf{only} $2s+1$ generalized moments.

\begin{theorem}\label{Exact reconstruction Theorem}
Let $\mathcal F$ be an homogeneous $M$-system on $I$. Consider a nonnegative measure $\sigma$ with
finite support included in $I$. \textbf{Then} the measure $\sigma$ is the \textbf{unique} solution
to generalized minimal extrapolation given  observation $\K_n(\sigma)$, where $n$ is not less than
twice the size
of the support of $\sigma$.
\end{theorem}
\begin{proof}
 The complete proof can be found in \ref{proof Exact reconstruction Theorem} but some key points
from \textit{the theory of approximation} are presented in \ref{Nonnegative}. For further insights
about \textit{Markov systems}, we recommend the books \cite{MR0458081,MR0204922}.
\end{proof}
\noindent In addition, this result is sharp in the following sense. Every measure with support
size $s$ depends on $2s$ parameters ($s$ for its support and $s$ for its weights).
Surprisingly, this information can be recovered from {only} $2s+1$ of its generalized
moments. Furthermore the program \eqref{support pursuit} does not use the fact that the target is
\textbf{nonnegative}. It recovers $\sigma$ among \textbf{all} \textbf{signed} measures with
finite support.
\subsubsection{Nonnegative interpolation}\label{Nonnegative}
An important property of $M$-systems is the existence of a nonnegative generalized polynomial
that vanishes exactly at a prescribed set of points $\{t_1,\dotsc,t_m\}$, where $t_i\in I$ for all
$i=1,\dotsc,m$. Indeed, define \textit{the index} as
\begin{equation}\label{Index}
 \Index(t_1,\dotsc,t_m)=\sum_{j=1}^m\chi(t_j)\,,
\end{equation}
where $\chi(t)=2$ if $t$ belongs to $\mathring I$ (the interior of $I$) and $1$ otherwise. The
next lemma guarantees the existence of nonnegative generalized polynomials.
\begin{lemma}[Nonnegative generalized polynomial]\label{nonnegative} Consider an $M$-system
$\mathcal F$ and points $t_1,\dotsc,t_m$ in $I$. These points are the \textbf{only} zeros
of a nonnegative generalized polynomial of degree at most $n$ \textbf{if and only if}
$\mathrm{Index}(t_1,\dotsc,t_m)\leq n$.
\end{lemma}
\noindent A proof of this lemma is in \cite{MR0458081}. Note that this lemma
holds
for \textbf{all} $M$-systems. However our main theorem needs a\textit{ homogeneous} $M$-system.

\subsubsection{Is \textbf{homogeneous} necessary?}\label{counter_example}
If one considers non-homogeneous $M$-systems then it is possible to give counterexamples that go
against Theorem \ref{Exact reconstruction Theorem} for all $n\geq2s$. Indeed, we have the next
result.
\begin{proposition}\label{Proposition Counter Example}
 Let $\sigma$ be a nonnegative measure supported by $s$ points. Let $n$ be an integer such
that
$n\geq2s$. Then there exists an $M$-system $\mathcal F$ and a measure $\mu\in\M$ such that
$\K_n(\sigma)=\K_n(\mu)$ and $\normTV\mu<\normTV\sigma$.
\end{proposition}
\begin{proof}
 See \ref{proof Proposition Counter Example}.
\end{proof}
Theorem \ref{Exact reconstruction Theorem} gives us the opportunity to build a large family of
\textit{deterministic} matrices for compressed sensing in the case of nonnegative signals.
\subsection{Deterministic matrices for compressed sensing}\label{DeterministicMatrices}
The heart of this article lies in the next theorem. It gives \textbf{deterministic} matrices for
\textit{compressed sensing}. We begin with some \textit{state-of-the-art} results in compressed
sensing. In the following, $p$ denotes the number of predictors (or, from a signal processing view
point, the length of the signal).
\begin{description}
 \item [Deterministic Design] As far as we know, for
\[
 n=\operatorname{\mathcal O}\displaylimits_{p,\,s\to\infty}\Big( s\,\log\Big(\frac p
s\Big)\Big)\,,
\]
there exists \cite{Ber} a {deterministic} matrix $A\in\R^{n\times p}$ such that basis pursuit
\eqref{BasisPursuit} recovers {all} $s$-sparse vectors from the observation $Ax_0$.
\item[Random Design]{If}
\[
 n\geq C\, s\,\log\Big(\frac p s\Big),
\]
where $C>0$ is a universal constant, then there exists (with high probability) a {random} matrix
$A\in\R^{n\times p}$ such that basis pursuit recovers {all} $s$-sparse vectors from the
observation $Ax_0$.
\end{description}
The deterministic result holds for \textit{large} values of $s$, $n$ and $p$. For readability we
do not specify the sense of \textit{large} here. The reader may
find an abundant literature in the respective references (see for example \cite{Ber,MR2241189}).

Considering nonnegative sparse vectors, it is possible to drop the
bound on $n$ to \[n\geq2s+1\,.\] Unlike
the above examples, this result holds for \textit{all} values of the parameters (as soon as
$n\geq2s+1$). In addition it give \textbf{explicit} design
matrices for basis pursuit. Last but not least, this bound
on $n$ does not depend on $p$. In special cases, this result has been previously developed in
\cite{donoho1992maximum,fuchs1996linear,donoho2005sparse,donoho2010counting}. Using Theorem
\ref{Exact reconstruction Theorem}, it is possible to provide a generalization of this result to a
broad range of measurement matrices:

\begin{theorem}[Deterministic Design Matrices]\label{Theorem Deterministic}
 Let $n,p,s$ be integers such that \[s \leq\min(n/2,\,p).\]
Let $\{1,u_1,\dotsc,u_n\}$ be
a homogeneous $M$-system on $I$. Let $t_1,\dotsc,t_p$ be distinct reals
of $I$. Let $A$ be generalized Vandermonde system defined by
\[A=
\begin{pmatrix}
    1 & 1 & \dotsc & 1\\
    u_1(t_1) & u_1(t_2) & \dotsc & u_1(t_p)\\
    u_2(t_1) & u_2(t_2) & \dotsc & u_2(t_p)\\
    \vdots & \vdots &  & \vdots\\
    u_n(t_1) & u_n(t_2) & \dotsc & u_n(t_p)
  \end{pmatrix}
\,.
\]
\textbf{Then} basis pursuit \eqref{BasisPursuit} exactly recovers \textbf{all nonnegative}
$s$-sparse vectors $x_0\in\R^p$ from the observation $Ax_0$.
\end{theorem}
\begin{proof}
See \ref{proof Theorem Deterministic}.
\end{proof}
\begin{remark}
 The purely analytical components of this result are tractable back to the theory of neighborly
polytopes (see for instance \cite{donoho2005sparse}) and in some sense trace to the theory of
moment problems which essentially follows from Carath\'eodory work
\cite{caratheodory1907variabilitetsbereich,caratheodory1911variabilitetsbereich}. Other relevant
work includes \cite{MR0059329,derry1956convex,Sturmfels:1988:TPM}. This list is not meant to be
exhaustive.
\end{remark}

\noindent Although the predictors could be highly correlated, basis pursuit
exactly recovers the target vector $x_0$. Of course, this result is theoretical. In practice, the
sensing matrix $A$ can be very ill-conditioned. In this case, basis pursuit behaves poorly.%

\subsubsection*{Numerical experiments}
Our numerical experiments illustrate Theorem \ref{Theorem Deterministic}. They are of the
following form:
\begin{enumerate}[$(a)$]
 \item Choose constants $s$ (sparsity), $n$ (number of known moments), and $p$ (length
of the vector). Choose the family $\mathcal F$ (cosine, polynomial, Laplace, Stieltjes,...).
\item Select the subset $\mathcal S$ (of size $s$) uniformly at random.
\item Randomly generate an $s$-sparse vector $x_0$ of support $\S$ whose nonzero entries have the
chi-square distribution with $1$ degree of
freedom.
\item Compute the observation $Ax_0$.
 \item Solve \eqref{BasisPursuit}, and compare with the target vector $x_0$.
\end{enumerate}
\medskip

\begin{figure}[ht]
\centering
\includegraphics[height =5.5cm]{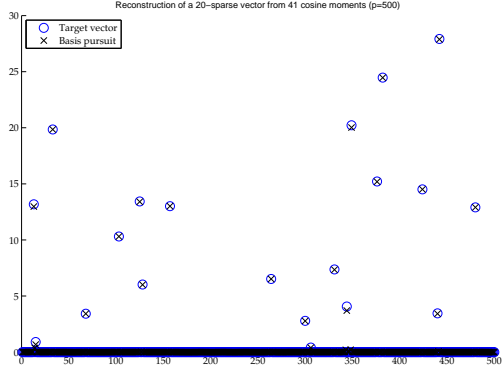}
\caption{Consider the family $\mathcal
F_{\cos{}}=\left\{1,\cos(\pi x),\cos(2\pi x),\dotsc\right\}$ on $I=[0,1]$ and the points
$t_k=k/501$, for $k=1,\dotsc,500$. The blue circles represent the target vector $x_0$ (a
$20$-sparse vector), while the black crosses represent the solution $x^\star$ of
\eqref{BasisPursuit} from the observation of $41$ cosine moments. In this example $s=20$, $n=41$,
and $p=500$. More numerical
results can be found in Appendix
\ref{Num}. This example shows that the reconstruction is excellent.}\label{ExpCos}
 \end{figure}

\noindent The program \eqref{BasisPursuit} can be recast as a linear program (see
\cite{MR1854649} for instance). Then we use an interior point method to solve
\eqref{BasisPursuit}.
\medskip

\noindent The entries of the target signal are distributed according to chi-square
distribution with $1$ degree of freedom. We chose this distribution to ensure that the entries are
nonnegative. Let us emphasize that the actual values of $x_0$ can be arbitrary; only the sign
matters. The result remains the same if we take the nonzero entries to be $1$,
say.
\medskip

\noindent
Let us denote $K:t\mapsto(1,u_1(t),\dotsc,u_n(t))$. The
columns of $A$ are the values of this map at points $t_1,\dotsc,t_p$. For large $p$, the
vectors $K(t_i)$ can be highly correlated. In fact, the
matrix $A$ can be \textit{ill-conditioned}. To avoid such
a case, we chose a family such that the map $K$ has a large derivative. It
appears that the \textit{cosine family} gives very good numerical results (see
Figure \ref{ExpCos}).
\medskip

\noindent We investigate the reconstruction error between the numerical result $\tilde x$ of the
program \eqref{BasisPursuit} and the target vector $x_0$. Our experiment is of the following form:
\begin{enumerate}[$(a)$]
 \item Choose $p$ (length of the vector) and $N$ (number of numerical experiments).
  \item Let $s$ satisfy $1\leq s\leq (p-1)/2$.
   \item Set $n=2s+1$ and solve the program \eqref{BasisPursuit}. Let $\tilde x$ be the numerical
result.
  \item Compute the $\ell_1$-error $\norm{\tilde x-x_0}_1/p$.
\item Repeat $N$ times the steps $(c)$ and $(d)$, and compute $\mathrm{Err}_s$, the arithmetic mean
of the $\ell_1$-errors.
\item Return $\norm{\mathrm{Err}_s}_\infty$, the maximal value of $\mathrm{Err}_s$.
\end{enumerate}
For $p=100$ and $N=10$, we find that
\[\norm{\mathrm{Err}_s}_\infty\leq0.05\,.\]
Note that all experiments were done for $n=2s+1$. This is the smallest value of $n$
such that Theorem \ref{DeterministicMatrices} holds.
\section{Exact reconstruction for generalized Chebyshev measures}\label{Tigre}
In this section we give some examples of extremal polynomials $P$ as they appear in Definition
\ref{DEF TYPE}. Considering $M$-systems, corollary of Lemma \ref{dual polynomial lemma} shows
that every measure with Jordan support included in $\big(\Eplus{P},\Eminus{P}\big)$ is the
\textbf{only} solution to \eqref{support pursuit}. Indeed, condition $(i)$ of Lemma \ref{dual
polynomial lemma} is clearly satisfied when the underlying family $\mathcal F$ is an $M$-system.

\subsection{Trigonometric families}
In the context of $M$-systems we can exhibit some very particular dual polynomials. The global
extrema of these polynomials gives families of support for which results of Lemma \ref{dual
polynomial lemma} hold.
\subsubsection*{The cosine family}
First, consider the $(n+1)$-dimensional cosine
system \[\mathcal F_{\cos{}}^n:=\left\{1,\cos(\pi x),\dotsc,\cos(n\pi x)\right\}\] on
$I=[0,1]$. Obviously, extremal polynomials \[P_k(x)=\cos(k\pi x),\] for $k=1,\dotsc,n$,
satisfy $\norm{P_k}_\infty\leq 1$ and $P_k(l/k)=(-1)^l$, for $l=0,1,\dotsc, (k-1)$. According to
Definition \ref{DEF TYPE}, let us denote
\begin{itemize}
 \item $\Eplus{P_k}:=\big\{2l/k\ |\ l=0,\dotsc,\lfloor\frac {k-1}2\rfloor\big\}$,
  \item $\Eminus{P_k}:=\big\{(2l-1)/k\ |\ l=1,\dotsc,\lfloor\frac k2\rfloor\big\}$.
\end{itemize}
The corollary that follows Lemma \ref{dual polynomial lemma} asserts the following result.
\medskip

\noindent \textit{Consider a signed measure $\sigma$ having Jordan support $(\mathcal S^+,\mathcal
S^-)$ such that
$\S^+\subset \Eplus{P_k}$ and $\S^-\subset \Eminus{P_k}$, for some $1\leq k\leq n$. \textbf{Then}
the measure $\sigma$ can be exactly reconstructed from the observation of
\begin{equation}\label{cos obs}
 \int_0^1\cos(k\pi t)\d \sigma(t),\quad k=0,1,\dotsc,n.
\end{equation}
Moreover, since the family  $\mathcal F_{\cos{}}^n$ is an $M$-system, condition $(i)$ in Lemma
\ref{dual polynomial lemma} is satisfied. Hence, the measure $\sigma$ is the only solution of
\eqref{support pursuit} given the observations \eqref{cos obs}.}
\medskip

\noindent Using the classical mapping
\[
\Psi:\bigg\{
\begin{array}{ccc}
 [0,1] & \to & [-1,1]\\
 x & \mapsto & \cos(\pi x)
\end{array}
\,,
\]
the system of function $(1,\cos(\pi x),\dotsc,\cos(n\pi x))$ can be push-forward to the system
of functions
$(1,T_1(x),\dotsc, T_n(x))$, where $T_k(x)$ is the so-called\textit{ Chebyshev
polynomial of the first kind} of order $k$, $k=1,\dotsc,n$ (see \ref{Cheby Sub}).
\subsubsection*{The characteristic function}
By the same token, consider the complex valued $M$-system defined by
\[\mathcal F_c^n=\left\{1,\exp(\imath\pi x),\dotsc,\exp( \imath n\pi x)\right\}\] on $I=[0,2)$. In
this case, one can check that
\[P_{\alpha,k}(t)=\cos(k\pi(t-\alpha)),\quad\forall t\in[0,2),\]
where $\alpha\in\R$ and $0\leq k\leq n/2$, is a generalized polynomial. Following the previous
example, we set
\begin{itemize}
 \item $\Eplus{P_{\alpha,k}}:=\big\{\alpha+2l/k\ (\mathrm{mod}\ 2)\ |\ l=0,\dotsc,\lfloor\frac
{k-1}2\rfloor\big\}$,
  \item $\Eminus{P_{\alpha,k}}:=\big\{\alpha +(2l-1)/k\ (\mathrm{mod}\ 2)\ |\
l=1,\dotsc,\lfloor\frac k2\rfloor\big\}$.
\end{itemize}
Hence Lemma \ref{dual polynomial lemma} can be applied. It yields the following:
\medskip

\noindent \textit{Any signed measure having
Jordan support included in $\big(\Eplus{P_{\alpha,k}},\Eminus{P_{\alpha,k}}\big)$,
for some $\alpha\in\R$ and $1\leq k\leq n/2$, is the \textbf{unique} solution of \eqref{support
pursuit} given the observation
\[\displaystyle\int\nolimits_{0}^2\exp(\imath k\pi
t)\d\sigma(t)=\varphi_{\sigma}(k\pi),\quad\forall k=0,\dotsc,n\,,\]
where $\varphi_\sigma(k\pi)$ has been defined in the previous section $($see \ref{Msystems}$)$.}

\noindent Note that the study of basis pursuit with this kind of trigonometric moments has been
considered in the pioneering work of Donoho and Stark \cite{MR997928}.
\subsection{Chebyshev polynomials}\label{Cheby Sub}
As mentioned in the introduction, the $k$-th
\textit{Chebyshev polynomial of the first order} is
defined by
\[
 T_k(x)=\cos(k\arccos(x)),\quad\forall x\in[-1,1]\,.
\]
We give some well known properties of Chebyshev polynomials. The $k$-th Chebyshev polynomial
satisfies the \textit{equioscillation
property} on $[-1, 1]$. In fact, there exist $k+1$ points $\zeta_i= \cos(\pi i/k)$
with $1=\zeta_0>\zeta_1>\dotsb>\zeta_k=-1$ such that
\[T_k(\zeta_i)=(-1)^i\norm{T_k}_\infty=(-1)^i\,,\]
where the supremum norm is taken over $[-1,1]$. Moreover, the Chebyshev polynomial $T_k$ satisfies
the
following \textit{extremal property}.
\begin{theorem}[\cite{MR1060735,MR1367960}] We have
\[\min_{p\in\mathcal
P_{k-1}^{\mathbb{C}}}\big\lVert{x^k-p(x)}\big\lVert_\infty=\big\lVert{2^{1-k}T_k}
\big\lVert_\infty=2^ { 1-k } \,,\]
 where $\mathcal P_{k-1}^{\mathbb{C}}$ denotes the set of complex polynomials of degree less
than $k-1$, and the supremum norm is taken over $[-1,1]$. Moreover, the minimum is uniquely
attained
by $p(x)=x^k-2^{1-k}T_k(x)$.
\end{theorem}
\noindent These two properties, namely the equioscillation property and the extremal property, will
be useful to us when we define generalized Chebyshev polynomial.

Using Lemma \ref{dual polynomial lemma} we uncover an exact reconstruction
result. Consider the family \[\mathcal F_p^n=\{1,x,x^2,\dotsc,x^n\}\] on $I=[-1,1]$. Set
\begin{itemize}
 \item  $\Eplus{T_k}=\big\{\cos(2l\pi/k),\
l=0,\dotsc,\big\lfloor\frac
k2\big\rfloor\big\}$,
  \item $\Eminus{T_k}=\big\{\cos((2l+1)\pi/k),\
l=0,\dotsc,\big\lfloor\frac k2\big\rfloor\big\}$.
\end{itemize}
The following result holds:
\medskip

\noindent \textit{Consider a signed measure $\sigma$ having Jordan support included in
$\big(\Eplus{T_k},\Eminus{T_k}\big)$, for some $1\leq k\leq n$. Then the measure $\sigma$ is the
\textbf{only} solution to \eqref{support pursuit} given its
first $(n+1)$ standard moments.}
\medskip

\noindent Note that this result is restrictive in the location of the support points, they are not
sparse in the usual sense, because they must be precisely located. Nevertheless, it can
be extended to any $M$-systems with the help of \textit{generalized
Chebyshev polynomials}.

\subsection{Generalized Chebyshev polynomials}
Following \cite{MR1367960}, we define generalized Chebyshev polynomials as follows. Let
$\mathcal F=\{u_0,u_1,\dotsc,u_n\}$ be an $M$-system on $I$.
\subsubsection{Definition}\label{Def Cheby}
The \textit{generalized Chebyshev polynomial}
\[\Che_k:=\Che_k\{u_0,u_1,\dotsc,u_n;I\}\,,\]
 where $1\leq k\leq n$, is defined by the following three properties:
\begin{itemize}
 \item $\Che_k$ is a generalized polynomial of degree $k$, i.e.
$\Che_k\in\mathrm{Span}\{u_0,u_1,\dotsc,u_k\}$,
  \item there exists $x_0<x_1<\dotsb<x_k$ such that
\begin{equation}\label{Che Altern}
 \sgn(\Che_k(x_{i+1}))=-\sgn(\Che_k(x_{i}))=\pm\norm{\Che_k}_\infty\,,
\end{equation}
for $i=0,1,\dotsc,k-1$,
\item and
\begin{equation}\label{Che unique}
 \norm{\Che_k}_\infty=1\quad\mathrm{with}\quad \Che_k( \max I)>0\,.
\end{equation}
\end{itemize}

\noindent The existence and the uniqueness of such $\Che_k$ is proved in \cite{MR1367960}.
Moreover, the following theorem shows that the extremal property implies the equioscillation
property \eqref{Che Altern}.
\begin{theorem}[\cite{MR1060735,MR1367960}]
 The $k$-th generalized Chebyshev polynomial $\Che_k$ exists and can be written as
\[\Che_k=c\,\Big(u_k-\sum_{i=0}^{k-1}a_i u_i\Big)\,,\]
where $a_0,a_1,\dotsc,a_{k-1}\in\R$ are chosen to minimize
\[\bigg\lVert{u_k-\sum_{i=0}^{k-1}a_i u_i}\bigg\lVert_\infty\,,\]
and the normalization constant $c\in\R$ can be chosen so that $\Che_k$ satisfies property
\eqref{Che unique}.
\end{theorem}
\noindent Generalized Chebyshev polynomials give a new family of extrema Jordan type measures
(see Definition \ref{DEF TYPE}). The corresponding target measures are named \textit{ Chebyshev
measures}.

\subsubsection{Exact reconstruction of Chebyshev measures}

\noindent Considering the equioscillation property \eqref{Che Altern}, set
\begin{itemize}
 \item $\Eplus{\Che_k}$ as the set of the alternation point $x_i$ such that
$\sgn(\Che_k(x_{i}))=\norm{\Che_k}_\infty$,
 \item $\Eminus{\Che_k}$ as the set of the alternation point $x_i$ such that
$\sgn(\Che_k(x_{i}))=-\norm{\Che_k}_\infty$.
\end{itemize}
\noindent A direct consequence of the last definition is the following proposition.

\begin{proposition}\label{Cheby Exact}
 Let $\sigma$ be a signed measure having Jordan support included in
$(\Eplus{\Che_k},\Eminus{\Che_k})$, for some $1\leq k\leq n$. \textbf{Then} $\sigma$ is the
\textbf{unique} solution to generalized minimal extrapolation \eqref{support pursuit} given
$\K_n(\sigma)$,
i.e. its $(n+1)$ first generalized moments.
\end{proposition}
\noindent In the special case $k=n$, Proposition \ref{Cheby Exact} shows that \eqref{support
pursuit} recovers all signed measures with Jordan support included in
$(\Eplus{\Che_n},\Eminus{\Che_n})$ from $(n+1)$ first generalized moments. Note that
$\Eplus{\Che_n}\cup\Eminus{\Che_n}$ has size $n$. Hence, this proposition shows that, among all
signed measure on $[-1,1]$, \eqref{support pursuit} can recover a signed measure of support size
$n$
from only $(n+1)$ generalized moments. In fact, any measure with Jordan
support included in $(\Eplus{\Che_n},\Eminus{\Che_n})$ can be uniquely defined by only $(n+1)$
generalized moments.

As far as we know, it is difficult to give the corresponding
generalized Chebyshev polynomials for a given family $\mathcal F=\{u_0,u_1,\dotsc,u_n\}$.
Nevertheless, Borwein, Erd\'elyi, and Zhang \cite{MR1299454} gives the explicit
form of $\Che_k$ for rational spaces (i.e. the Stieltjes transformation in our framework). See
also \cite{MR997928, MR1421165} for some applications in \textit{optimal design}.

\subsubsection{Construction of Chebyshev polynomials for Stieltjes transformation}
We consider the case of Stieltjes transformation described in Section \ref{SRPM}. In this
case, Chebyshev polynomials $\Che_k$ can be precisely described. Consider homogeneous
$M$-system on $[-1,1]$ defined by
\[\widetilde{\mathcal F}_s^n=\Big\{1,\frac1{z_1-x},\frac1{z_2-x},\dotsc,\frac1{z_n-x}\Big\}\,,\]
where $(z_i)_{i=1}^k\subset\mathbb C\setminus[-1,1]$.

Reproducing \cite{MR1367960}, we can construct generalized Chebyshev polynomials of the first
kind. It yields
\[\Che_k(x)=\frac12(f_k(z)+f_k(z)^{-1}),\quad\forall x\in[-1,1]\,,\]
where $z$ is uniquely defined by $x=\frac12(z+z^{-1})$ and $\abs z<1$, and $f_k$ is a known
analytic function in a neighborhood of the closed unit disk. Moreover this
analytic function can be expressed in terms of only $(z_i)_{i=1}^k$. We refer to \cite{MR1367960}
for further details.

\section{The nullspace property for measures}\label{NSPSection}
In this section we consider \textbf{any} countable family $\mathcal F=\{u_0,u_1,\dotsc,u_n\}$ of
\textit{continuous functions} on $I$. In particular we do not assume that $\mathcal F$ is a
non-homogeneous $M$-system. We aim at deriving a sufficient condition for exact reconstruction of
signed measures. More precisely, we are concerned with giving a related property to the
\textit{nullspace property} \cite{MR2449058} of compressed sensing.

Note that the solutions to program \eqref{support pursuit} depend only on the first
$(n+1)$ elements of $\mathcal F$ and on the target measure $\sigma$. We investigate
the condition that the family $\mathcal F$ must satisfy to ensure exact reconstruction. In the
meantime, Cohen, Dahmen and DeVore introduced \cite{MR2449058} a relevant condition, the
\textit{nullspace property}. Their property binds the geometry of the
nullspace of $A$ and \textit{the best $k$-term approximation} of the target $x_0$ given the
observation $Ax_0$. This well known property can be stated as follows.
\subsection{The nullspace property in compressed sensing}
Let $A\in\R^{n\times p}$ be a matrix. We say that $A$ satisfies the nullspace property of order $s$
\textit{if and only if} for all nonzero vectors $h$ in the nullspace of $A$, and all subsets of
entries $S$ of size $s$, \[\norm{h_S}_1<\norm{h_{S^c}}_1\,,\] where $h_S$ denotes the
vector whose $i$-th entry is $h_i$ if $i\in S$ and $0$
otherwise. It is now standard that basis pursuit \eqref{BasisPursuit} exactly recovers all
$s$-sparse vectors $x_0$ (i.e. vectors with at most $s$ nonzero entries) \textit{if and only if}
the design matrix $A$ satisfies the nullspace property of order $s$.

In this section, we show that the same property holds for generalized minimal extrapolation.
According to the
compressed sensing literature, we keep the same name for this related property.
\subsection{The nullspace property for generalized minimal extrapolation}\label{NSP}
Consider the linear map $\K_n:\mu\mapsto(c_0(\mu),\dotsc,c_{n}(\mu))$ from $\M$ to $\R^{n+1}$. We
refer to this map as the
\textit{generalized moment morphism}. Its nullspace $\ker(\K_n)$ is a linear subspace of $\M$. The
\textit{Lebesgue decomposition theorem} is the precious tool used to define the nullspace property.
\subsubsection{The $S$-atomic part}
Let $\mu\in\M$ and $S=\{x_1,\dotsc,x_s\}$ be a finite subset of $I$. Define
$\Delta_S=\sum_{i=1}^s\delta_{x_i}$ as the \textit{Dirac comb} with support $S$. The Lebesgue
decomposition of $\mu$ with respect to $\Delta_S$ gives
\begin{equation}\label{decomposition}
 \mu=\mu_S+\mu_{S^c}\,,
\end{equation}
where $\mu_S$ is a \textbf{discrete} measure whose support is included in $S$, and $\mu_{S^c}$
is a measure whose support is included in $S^c:=I\setminus S$.
\subsubsection{The nullspace property with respect to a Jordan support family}
First, as in the standard compressed sensing context \cite{MR2449058}, we define the
nullspace property with respect to a Jordan support family $\varUpsilon$. This property is only a
sufficient condition for exact reconstruction of finite measure; see Proposition
\ref{nullspaceproperty}.
\begin{definition}[Nullspace property with respect to a Jordan support family $\varUpsilon$]
We say that the generalized moment morphism $\K_n$ satisfies the nullspace property with respect to
a Jordan support family $\varUpsilon$ \textbf{if and only if} it satisfies the
following property. For all nonzero
measures $\mu$ in the nullspace of $\K_n$, and for all $(\S^+,\S^-)\in\varUpsilon$,
\begin{equation}\label{Spread}
 \normTV{\mu_S}<\normTV{\mu_{S^c}},
\end{equation}
where $\S=\S^+\cup\S^-$.

\noindent \textbf{---} The \textbf{weak} nullspace property states as follows: For all nonzero
measures $\mu$ in the nullspace of $\K_n$, and for all $(\S^+,\S^-)\in\varUpsilon$,
\begin{equation*}
 \normTV{\mu_S}\leq\normTV{\mu_{S^c}},
\end{equation*}
where $\S=\S^+\cup\S^-$.
\end{definition}
\noindent Given a nonzero measure $\mu$ in the nullspace of $\K_n$, this property means that more
than half of the total variation of $\mu$ cannot be concentrated on a small subset. The nullspace
property is a key to exact reconstruction as shown in the following proposition.
\begin{proposition}\label{nullspaceproperty}
Let $\varUpsilon$ be a
Jordan support family. Let $\sigma$ be a signed measure having a Jordan
support in $\varUpsilon$. \textbf{If} the generalized moment morphism $\K_n$ satisfies the
nullspace property with respect to $\varUpsilon$, \textbf{then}, the measure $\sigma$ is the
\textbf{unique} solution of generalized minimal extrapolation \eqref{support pursuit} given the
observation $\K_n(\sigma)$.

\noindent \textbf{---} \textbf{If} the generalized moment morphism $\K_n$ satisfies the
\textbf{weak} nullspace property with respect to $\varUpsilon$, \textbf{then}, the measure $\sigma$
is \textbf{a} solution of generalized minimal extrapolation \eqref{support pursuit} given the
observation $\K_n(\sigma)$.
\end{proposition}
\begin{proof}
See \ref{proof nullspaceproperty}.
\end{proof}
\noindent As far as we know, it is difficult to check the nullspace property. In the following, we
give an example such that the weak nullspace property is satisfied.

\subsection{The spaced out interpolation}
We recall that $S_\Delta$ is the set of all pairs $(S^+,S^-)$ of subsets of $I=[-1,1]$
such that
\begin{equation}\label{S Delta}
 \forall x,y\in S^+\cup S^-,\ x\neq y,\quad \abs{x-y}\geq\Delta.
\end{equation}
The next lemma shows that if $\Delta$ is large enough then there exists a polynomial of degree $n$,
with supremum norm not greater than $1$, that interpolates $1$ on the set $\S^+$ and $-1$ on
the set $\S^-$.
\begin{lemma}\label{Delta}
For all $(S^+,S^-)\in S_\Delta$, there exists a
polynomial $P_{(S^+,S^-)}$ such that
\begin{itemize}
  \item $P_{(S^+,S^-)}$ has degree $n$ not greater than $(2/\sqrt{\pi})\, (\sqrt
e/\Delta)^{5/2+1/\Delta}$,
 \item $P_{(S^+,S^-)}$ is equal to $1$ on the set $S^+$,
\item $P_{(S^+,S^-)}$ is equal to $-1$ on the set $S^-$,
\item and $\lVert {P_{(S^+,S^-)}}\lVert_\infty\leq1$ over $I$.
\end{itemize}
\end{lemma}
\begin{proof}
See \ref{proof Delta}.
\end{proof}

\noindent This upper bound is meant to show that one can interpolate any sign sequence on
$S_\Delta$. Let us emphasize that this result is far from being sharp. Considering
$L_2$-minimizing polynomials under fitting constraint, the authors of the present paper believe
that one can greatly improve the upper bound of Lemma \ref{Delta}. Indeed, our
numerical experiments are in complete agreement with this comment. Invoking Lemma \ref{dual
polynomial lemma}, Lemma \ref{Delta} gives the next proposition.
\begin{proposition}\label{NSP Delta}
Let $\Delta$ be a positive real. \textbf{If} $n\geq (2/\sqrt{\pi})\, (\sqrt
e/\Delta)^{5/2+1/\Delta}$ \textbf{then}
$\K_n$ satisfies the weak nullspace property with respect to $S_\Delta$.
\end{proposition}
\begin{proof}
See \ref{proof NSP Delta}.
\end{proof}

\noindent The bound $(2/\sqrt{\pi})\, (\sqrt e/\Delta)^{5/2+1/\Delta}$ can be considerably improved
in actual practice. The following numerical experiment shows that this bound can be greatly
lowered.

\subsubsection*{Some simulations}

\begin{figure}[ht]
\centering
\includegraphics[height =5cm]{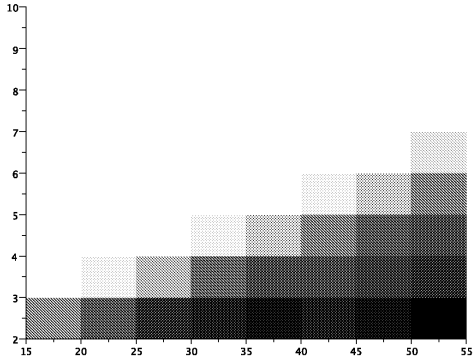}
\caption{Consider the family $\mathcal F_{\cos{}}=\left\{1,\cos(\pi x),\cos(2\pi
x),\dotsc\right\}$ on $I=[0,1]$. Set $s=10$ the size of the target support. We are concerned with
signed measures with Jordan support in $S_\Delta$ (see \eqref{S Delta}). The abscissa represents
the values of $1/\Delta$ (with $\Delta=1/15,1/20,\dotsc,1/55$), and the ordinates
represent the values of $n$ (with $n=20,30,\dotsc,100$). For each value of $(\Delta,n)$, we draw
uniformly  $100$ realizations of signed measures and the corresponding $L_2$-minimizing
polynomial $P$. The gray scale represents the percentage of times that $\norm
P_\infty\leq1$ occurs. The white color means $100\%$ (\eqref{support pursuit} exactly recovers
\textbf{all}
the signed measures) while the black color represent $0\%$ (in all our experiments, the polynomial
$P$ is such that $\norm P_\infty>1$ over $I$).
}\label{ExpPoly}
 \end{figure}

\noindent Our numerical experiment consists in looking for a generalized polynomial satisfying the
assumption
of Lemma \ref{dual polynomial lemma}. We work here with the cosine system $(1, \cos(\pi
x),\cos(2\pi x),\dotsc, \cos(n\pi x))$ for various values of the integer $n$. As explained in
Section \ref{Tigre}, we can also consider the more classical power system $(1,
x,x^2,\dotsc, x^n)$, so that our numerical experiments may be interpreted in this
last frame.  We consider  signed measure having a support $\mathcal{S}$ with $|\mathcal{S}|=10$. We
consider $\Delta$-spaced out type measures for various values of $\Delta$. For each choice of
$\Delta$, we draw uniformly  $100$ realizations of signed measures. This means that the points
of $\mathcal{S}$ are uniformly drawn on $I^{10}$, where $I=[0,1)$ here, with the restriction that
the minimal distance
between two points is at least $\Delta$ and that there exists two of points that are exactly
$\Delta$ away from each other. Further, we uniformly randomized the signs of the
measure on each point of $\mathcal{S}$. As we wish to work with {\it true} signed measures, we do
not allow the case where all the signs are the same (negative or positive measures). Once we
simulated the set $\mathcal{S}^+$ and $\mathcal{S}^-$, we wish to build an interpolating
polynomial $P$ of degree $n$ having value $1$ on $\mathcal{S}^+$, $-1$ on $\mathcal{S}^-$ and
having a supremum norm minimum. As this last minimization is not obvious, we relax it to the
minimization of the $L_2$-norm with the extra restriction that the derivative of the
interpolation polynomial vanishes on $\mathcal{S}$. Hence, when this last optimization problem has
a
solution having a supremum norm not greater than $1$, Lemma \ref{dual polynomial lemma} may be
applied and \eqref{support pursuit} leads to exact reconstruction. The proportion of experimental
results,
where the supremum norm of the $L_2$ optimal polynomial is not greater than $1$, is
reported in Figure \ref{ExpPoly}.

In our experiments we consider the values $\Delta=1/15,1/20,\dotsc,1/55$. According to
Proposition \ref{NSP Delta}, the corresponding values of $n$ range from $10^{19}$ to $10^{59}$. In
our experiments, we find that $n=80$ suffices.
\bigskip

{\noindent \textbf{Acknowledgements:} The authors would like to thank
Jean Paul Calvi and Viet Hung Pham for fruitful comments. We also thank anonymous referees for
their careful reviews and their interesting suggestions.
}
 \appendix

\section{Proofs of Section \ref{GDP}.}

\subsection{Proof of Lemma \ref{dual polynomial lemma}}\label{proof dual polynomial
proposition}
Assume that a generalized dual polynomial $P$ exists. Let $\sigma$ be such
that $\sigma=\sum_{i=1}^s\,\sigma_i\,\delta_{x_i}$, with $\sgn(\sigma_i)=\varepsilon_i$. Let
$\sigma^\star$ be a solution of the generalized minimal extrapolation \eqref{support pursuit} then
$\int
P\,\d\sigma=\int P\,\d\sigmastar$. The equality $(ii)$ yields $\normTV\sigma=\int
P\,\d\sigma$.
 Combining the two previous equalities,
\[\normTV\sigma=\int P\,\d\sigma=\int
P\,\d\sigmastar=\sum_{i=1}^s\varepsilon_i\,\sigmastar_i+\int
P\,\d\sigmastar_{\S^c}\,,\]
where $\varepsilon_i=\sgn(\sigma_i)$ and
\[\sigmastar=\sum_{i=1}^s\,\sigmastar_i\,\delta_{x_i}+\sigmastar_{\S^c}\,,\] according to the
Lebesgue decomposition \eqref{decomposition}.
Since $\norm{P}_\infty=1$, we have
\[\sum_{i=1}^s\varepsilon_i\,\sigmastar_i+\int
P\,\d\sigmastar_{\S^c}\leq\normTV{\sigmastar_\S}+\normTV{
\sigmastar_{\S^c}}=\normTV{\sigmastar}.\]
Observe $\sigmastar$ is a solution of \eqref{support pursuit}, it follows that
$\normTV\sigma=\normTV{\sigmastar}$ and the above inequality is an equality.
It yields $\int P\,\d\sigmastar_{\S^c}=\normTV{\sigmastar_{\S^c}}$. Moreover we have the following
result.
\begin{lemma}
Let $\nu\in\M$ with its support included in $\S^c$.  \textbf{If} $\int
P\,\d\nu=\normTV{\nu}$ \textbf{then} $\nu=0$.
\end{lemma}
\begin{proof}
 Consider the compact set
\[\Omega_k=I\setminus\bigcup_{i=1}^s\,\big]x_i-\tfrac 1 k,\,x_i+\tfrac1 k\big[,\quad\forall k>0,\]
Suppose that there exists $k>0$ such that $\normTV{\nu_{\Omega_k}}\neq0$. Then the inequality
$(iii)$ leads to $\int_{\Omega_k}P\,\d\nu<\normTV{\nu_{\Omega_k}}$. It yields
\[\normTV{\nu}=\int
P\,\d\nu=\int_{\Omega_k}P\,\d\nu+\int_{\Omega_k^c}P\,\d\nu<\normTV{\nu_{\Omega_k}}+\normTV{\nu_{
\Omega_k^c}}=\normTV{\nu}\,,\]
which is a contradiction. We deduce that $\normTV{\nu_{\Omega_k}}=0$, for all $k>0$. The equality
$\nu=0$ follows
with ${\S^c}=\cup_{k>0}\Omega_k$ .
\end{proof}
\noindent This lemma shows that $\sigmastar$
is a discrete measure with its support included in $\S$. In this case, the moment constraint
$\K_n(\sigmastar-\sigma)=0$ can be written as a \textit{generalized
Vandermonde system},
\[
\begin{pmatrix}
    u_0(x_1) & u_0(x_2) & \dotsc & u_0(x_s)\\
    u_1(x_1) & u_1(x_2) & \dotsc & u_1(x_s)\\
    \vdots & \vdots &  & \vdots\\
    u_{n}(x_1) & u_{n}(x_2) & \dotsc & u_{n}(x_s)
  \end{pmatrix}
\begin{pmatrix}
      \sigmastar_1-\sigma_1\\
      \sigmastar_2-\sigma_2\\
    \vdots\\
   \sigmastar_s-\sigma_s
  \end{pmatrix}=0\,.
\]
From condition $(i)$, we deduce that the generalized
Vandermonde system is injective. \hfill$\square$

\subsection{Proof of the remark in Section \ref{subgradient}}\label{proof subgradient}
Let $\sigma$ belong to $\mathcal
F\big(x_1,\varepsilon_1,\dotsc,x_s,\varepsilon_s\big)$. Consider
the linear functional,
\[\Phi_f:\mu\mapsto\int_I f\,\d\mu,\]
where $f$ denotes a continuous bounded function.
By definition, any subgradient $\Phi_f$ of the $TV$-norm at point $\sigma$ satisfies, for all
measures
$\mu\in\M$,
\[\normTV\mu-\normTV\sigma\geq\Phi_f(\mu-\sigma).\]
Thus, one can easily check that $f$ is equal to $1$ (resp. $-1$) on $\mathrm{supp}(\sigma^+)$
(resp. $\mathrm{supp}(\sigma^-)$) and that $\norm f_\infty=1$. Conversely, any function
$f$ satisfying the latter condition leads to a subgradient $\Phi_f$. Therefore, when it exists, the
generalized dual polynomial $P$ is such that $\Phi_P$ is a
\textit{subgradient} of the $TV$-norm at point $\sigma$. Furthermore, let $\mu$ be
a feasible point (i.e. $\K_n(\mu)=\K_n(\sigma)$). Since $P$ is a generalized polynomial of order
$n$, we deduce that $\Phi_P(\mu-\sigma)=0$. Hence, the subgradient $\Phi_P$ is perpendicular to
the set of feasible points.\hfill$\square$
\section{Proofs of Section \ref{SRPM}}
\subsection{Proof of Theorem \ref{Exact reconstruction Theorem}}\label{proof Exact reconstruction
Theorem}
The proof essentially relies on Lemma \ref{dual polynomial lemma}.
Let $s$ be an integer. Let $\sigma$ be
a nonnegative measure. Let $\S=\{x_1,\dotsc,x_s\}\subset I$ be its support. The next
lemma
shows the existence of a generalized dual polynomial.
\begin{lemma}[Dual polynomial]\label{Existence dual polynomial}
Let $s$ be an integer and $n$ be such that $n=2s$. Let $\mathcal F$ be a homogeneous
$M$-system on $I$. Let $(x_1,\dotsc,x_s)$ be such that $\mathrm{Index}(x_1,\dotsc,x_s)\leq n$. Then
there exists a generalized polynomial $P$ of degree $d$ such that
 \begin{enumerate}[$(i)$]
  \item $s\leq d\leq n$,
  \item $P(x_i)=1,\ \forall\,i=1,\dotsc,s\,,$
  \item $\abs{P(x)}<1$ for all $x\notin\{x_1,\dotsc,x_s\}$.
\end{enumerate}
\end{lemma}
\noindent We recall that $\mathrm{Index}$ is defined by \eqref{Index}. Note that these
polynomials are presented in the first example of Definition \ref{DEF TYPE}.
\begin{proof}[Proof of Lemma \ref{Existence dual polynomial}]
Let $(x_1,\dotsc,x_s)$ be such that $\mathrm{Index}(x_1,\dotsc,x_s)\leq n$. From Lemma
\ref{nonnegative}, there exists a nonnegative polynomial $Q$ of degree $d$ that vanishes
exactly at the points $x_i$. Moreover, its degree $d$ satisfies $(i)$.

Since $Q$ is continuous on the compact set $\overline I$, it is bounded and there exists a real
$c$ such that $\norm Q_\infty<1/c$.
The generalized polynomial $P=1-cQ$ is the expected generalized polynomial.
\end{proof}
\noindent Observe that
\begin{itemize}
 \item Using Lemma \ref{Existence dual polynomial}, it yields that there exists a generalized dual
polynomial, of degree at most $n=2s$, which interpolates the value $1$ at points
$\{x_1,\dotsc,x_s\}$.
\item Since $\mathcal F=\{u_0,u_1,\dotsc,u_n\}$ is a $T$-system, the Vandermonde system given
by $(i)$ in Lemma \ref{dual polynomial lemma} has full column rank.
\end{itemize}
Invoke Lemma \ref{dual polynomial lemma} to conclude.\hfill$\square$

\begin{remark}
Since $\mathcal F$ is a homogeneous $M$-system, the constant function $1$ is a generalized
polynomial. Note that the linear combination $P=1-cQ$ is a generalized polynomial because $1$ is a
generalized polynomial.  This assumption is essential $($see \ref{counter_example}$)$.
\end{remark}

\subsection{Proof of Proposition \ref{Proposition Counter Example}}\label{proof Proposition
Counter Example}
 Let $\sigma=\sum_{i=1}^s\sigma_i\delta_{x_i}$ be a nonnegative measure.
Let $\mathcal S=\{x_1,\dotsc,x_s\}$ be its support. Let $n$ be an integer such that
$n\geq2s$.
\subsubsection*{Step 1:}  Let $\mathcal
F_h=\{1,u_1,u_2,\dotsc\}$ be a homogeneous $M$-system (the standard polynomials for instance).
Let $t_1,\dotsc,t_{n+1}\in I\setminus\mathcal S$ be distinct points. It follows that the
Vandermonde system $\Bigg(\begin{smallmatrix}
    1 & \dotsc & 1\\
    u_1(t_1) & \dotsc & u_1(t_{n+1})\\
    : & & :\\
    u_n(t_1) & \dotsc & u_n(t_{n+1})
  \end{smallmatrix}\Bigg)$ has full rank. Hence we may choose
$(\nu_1,\dotsc,\nu_{n+1})\in\R^{n+1}$
such that
\begin{itemize}
 \item $\nu=\displaystyle\sum_{i=1}^{n+1}\nu_i\delta_{t_i}$,
  \item and for all $k=0,\dotsc,n$, $\displaystyle\int_I u_k\d\nu=\int_I u_k\d\sigma$.
\end{itemize}

\subsubsection*{Step 2:} Set
\[r=\frac{\normTV\sigma}{\normTV\nu+1}\,.\]
Consider a positive continuous function $u_0$ such that
\begin{itemize}
 \item $u_0(x_i)=r$, for $i=1,\dotsc,s$,
\item $u_0(t_i)=1$, for $i=1,\dotsc,n+1$,
\item the function $u_0$ is not constant.
\end{itemize}
Set $\mathcal F=\{u_0,u_0\,u_1,u_0\,u_2,\dotsc\}$. Obviously, $\mathcal F$ is a
non-homogeneous $M$-system. As usual, let $\K_n$ denote the generalized moment morphism of order
$n$
derived from the family $\mathcal F$.
\subsubsection*{Last step:} Set
$\mu=r\,\nu$. An easy calculation gives $\K_n(\sigma)=\K_n(\mu)$. Note
that \[{\displaystyle\normTV\mu=\sum_{i=1}^{n+1}r\abs{\nu_i}=\frac{\sum\displaylimits_{i=1}^{n+1}
\abs{
\nu_i } } { \sum_ { i=1 } ^ { n+1 } \abs {
\nu_i}+1}\normTV\sigma<\normTV\sigma\,.}\]
\hfill$\square$

\subsection{Proof of Theorem \ref{Theorem Deterministic}}\label{proof Theorem Deterministic}
Set $\mathcal T=\{t_1,\dotsc,t_p\}$. Let  $\M_\T$ denote the set of all finite
measures of which support is included in $\T$. Let $\varTheta_\T$ be the linear
map defined by
\[\isoT:\Bigg\{\begin{array}{ccc}
         (\R^p,\,\ell_1) & \to & (\M_\T,\,\normTV{.})\\
	  (x_1,\dotsc,x_p) & \mapsto & \sum\limits_{i=1}^px_i\,\delta_{t_i}
        \end{array}\,.
\]
One can check that $\isoT$ is a bijective isometry. Moreover, it holds that
\begin{equation}\label{constraints}
 \forall y\in\R^p,\quad \K_n(\isoT(y))=Ay,
\end{equation}
where $A$ is the generalized Vandermonde system defined by
\[A=
\begin{pmatrix}
    1 & 1 & \dotsc & 1\\
    u_1(t_1) & u_1(t_2) & \dotsc & u_1(t_p)\\
    u_2(t_1) & u_2(t_2) & \dotsc & u_2(t_p)\\
    \vdots & \vdots &  & \vdots\\
    u_n(t_1) & u_n(t_2) & \dotsc & u_n(t_p)
  \end{pmatrix}
\,.
\]
In the meantime, let $x_0$ be a nonnegative $s$-sparse vector. Let
$\sigma=\isoT(x_0)$. Observe that the support size of $\sigma$ is at most $s$. Consequently,
Theorem \ref{Exact reconstruction Theorem} shows that $\sigma$ is the unique solution to
\eqref{support pursuit}. Since $\sigma\in\M_T$, we have that $\sigma$ is the unique solution to the
following program:
\[\sigma=\argmin\displaylimits_{\mu\in\M_\T}\normTV{\mu}\quad \text{s.t.} \
\K_n(\mu)=\K_n(\sigma)\,.\]
Using \eqref{constraints} and the isometry $\isoT$, it follows that $x_0$ is the unique solution to
the program:
\[x_0=\argmin\displaylimits_{y\in\R^p}\norm{y}_1\quad \text{s.t.} \
Ay=Ax_0\,.\]
\hfill$\square$
\section{Proofs of Section \ref{NSPSection}}

 \subsection{Proof of Proposition \ref{nullspaceproperty}}\label{proof nullspaceproperty}
 Let $\K_n$ be a generalized moment morphism that satisfies the nullspace
 property with respect to a Jordan support family $\varUpsilon$. Let $\sigma$ be a signed measure
of which Jordan support belongs to $\varUpsilon$. Let $\sigma^\star$ be a solution of
\eqref{support pursuit}. Observe that
$\lVert{\sigma^\star}\lVert_{TV}\,\leq\normTV\sigma$. Denote
$\mu=\sigma^\star-\sigma$ and note that $\mu\in\ker(\K_n)$. Then
 \begin{eqnarray*}
 \normTV{\sigmastar}& = & \normTV{\sigmastar_\S}+\normTV{\sigmastar_{\S^c}},\\
  & = & \normTV{\sigma+\mu_\S}+\normTV{\mu_{\S^c}},\\
  & \geq & \normTV{\sigma} - \normTV{\mu_\S}+\normTV{\mu_{\S^c}},
 \end{eqnarray*}
 where $\S$ denotes the support of $\sigma$. Suppose that $\mu\neq0$. The nullspace property
 yields that the measure $\mu$ satisfies inequality \eqref{Spread}. We
 deduce $\normTV{\sigmastar}>\normTV{\sigma}$, which is a contradiction.
 Thus $\mu=0$ and $\sigma^\star=\sigma$. \hfill$\square$

\subsection{Proof of Lemma \ref{Delta}}\label{proof Delta}
For sake of readability, we sketch the proof here. Let $(S^+,S^-)\in S_\Delta$. Set
$S=S^+\cup
S^-=\{x_1,\dotsc,x_s\}$.
Consider the Lagrange interpolation polynomials
\[l_k(x)=\frac{\prod_{i\neq k}(x-x_i)}{\prod_{i\neq k}(x_k-x_i)}\,,\]
for $1\leq k\leq s$. One can bound the supremum norm of $l_k$ over $[0,1]$ by
\[\norm{l_k}_\infty\leq L(\Delta),\]
where $L(\Delta)$ is an upper bound that depends only on $\Delta$. Consider the $m$-th
\textit{Chebyshev polynomial of the first order} $T_m(x)=\cos(m\arccos(x))$, for all $x\in[-1,1]$.
For a sufficiently large value of $m$, there exist $2s$ extrema $\zeta_i$ of $T_m$ such that
$\abs{\zeta_i}\leq 1/(s L(\Delta))$.  Interpolating values $\zeta_i$ at point $x_k$, we build
the expected polynomial $P$. We find that the polynomial $P$ has degree not greater than 
\[C\,(\sqrt e/\Delta)^{5/2+1/\Delta}\,,\]
where $C=2/\sqrt{\pi}$.\hfill$\square$

\subsection{Proof of Proposition \ref{NSP Delta}}\label{proof NSP Delta}
Let $\mu$ be a nonzero measure in the nullspace of $\K_n$ and $(\mathcal A,\mathcal B)$ be in
$S_\Delta$. Let $\S$ be equal to $\mathcal A\cup \mathcal B$. Let $\S^+$ (resp. $\S^-$) be the set
of
points $x$ in $\S$ such that the $\mu$-weight at point $x$ is nonnegative (resp. negative). Observe
that $\S=\S^+\cup\S^-$ and $(\S^+,\S^-)\in S_\Delta$. From Lemma \ref{Delta}, there exists
$P_{(S^+,S^-)}$ of degree not greater than $n$
such that $P_{(S^+,S^-)}$ is equal to $1$ on $S^+$, $-1$ on $S^-$, and $\lVert
{P_{(S^+,S^-)}}\lVert_\infty\leq1$. It yields
\[\int P_{(S^+,S^-)}\d\mu=\normTV{\mu_S}+\int_{S^c}
P_{(S^+,S^-)}\d\mu\geq\normTV{\mu_S}-\normTV{\mu_{S^c}}.\]
Since $\mu\in\mathrm{ker}(\K_n)$, it follows that $\int P_{(S^+,S^-)}\d\mu=0$.\hfill$\square$
\newpage
\section{Numerical Experiments}\label{Num}
\begin{figure}[ht]
 \centering
  \includegraphics[height = 4.5cm]{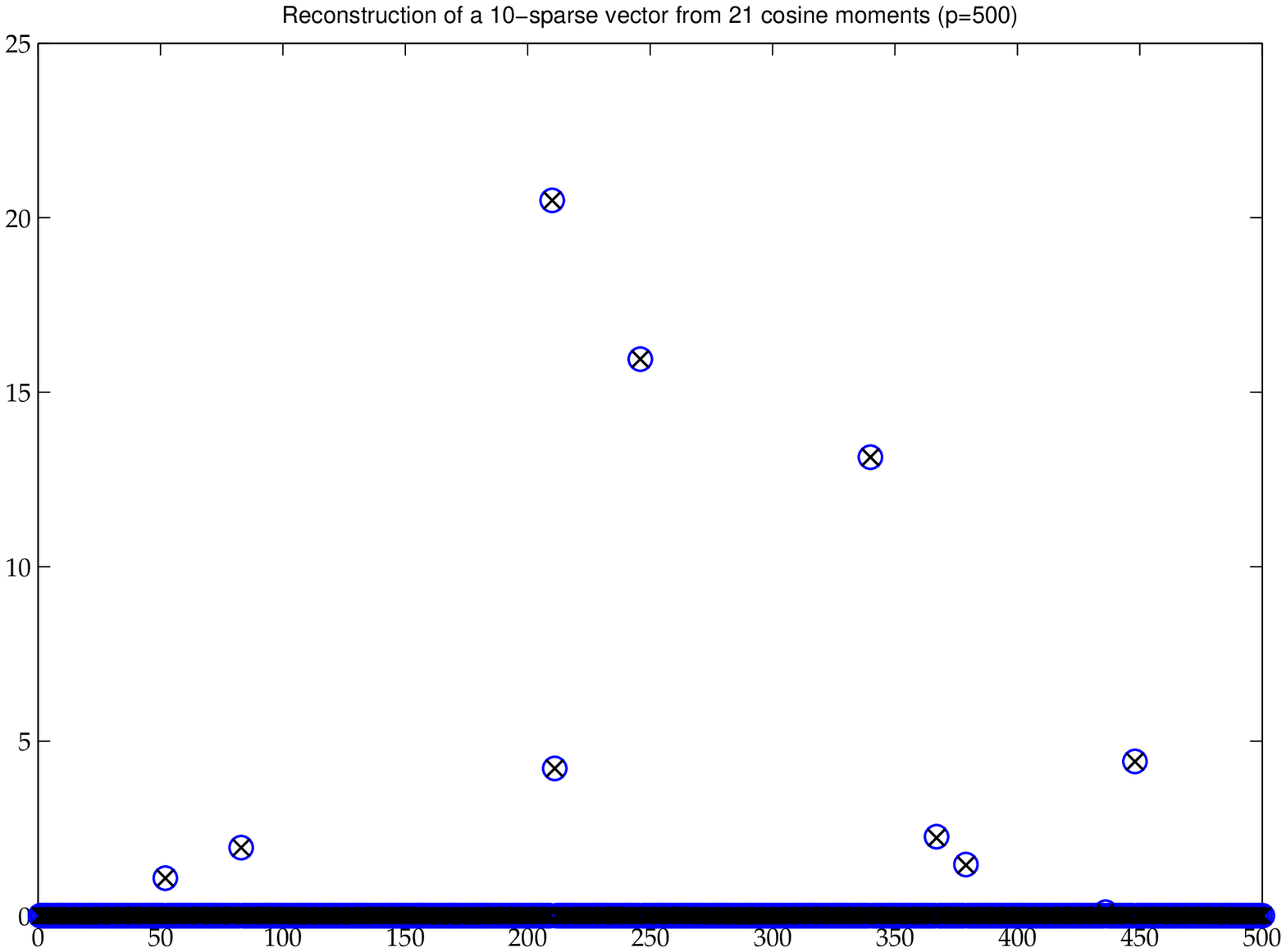}
\includegraphics[height =4.5cm]{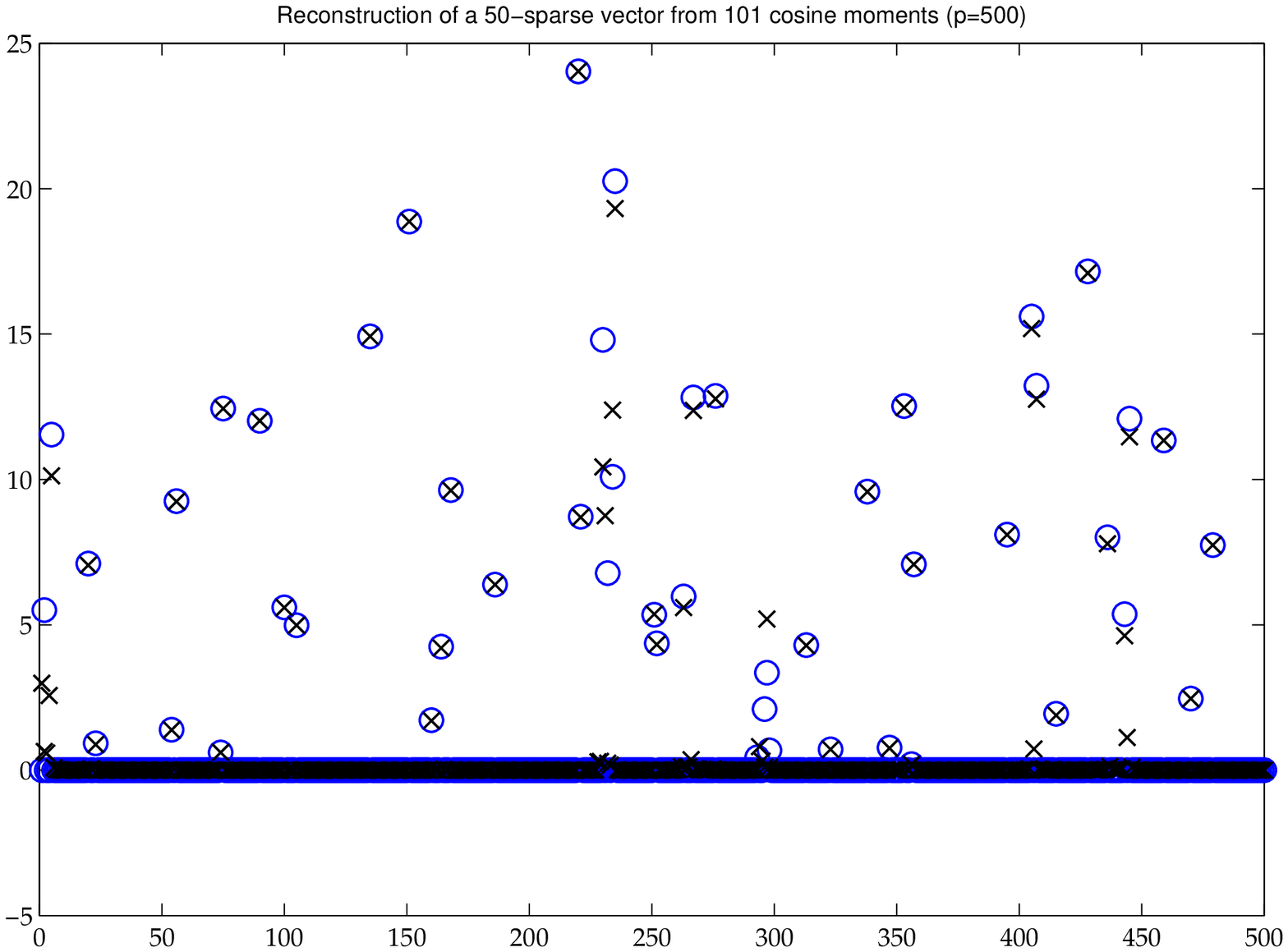}
\includegraphics[height =4.5cm]{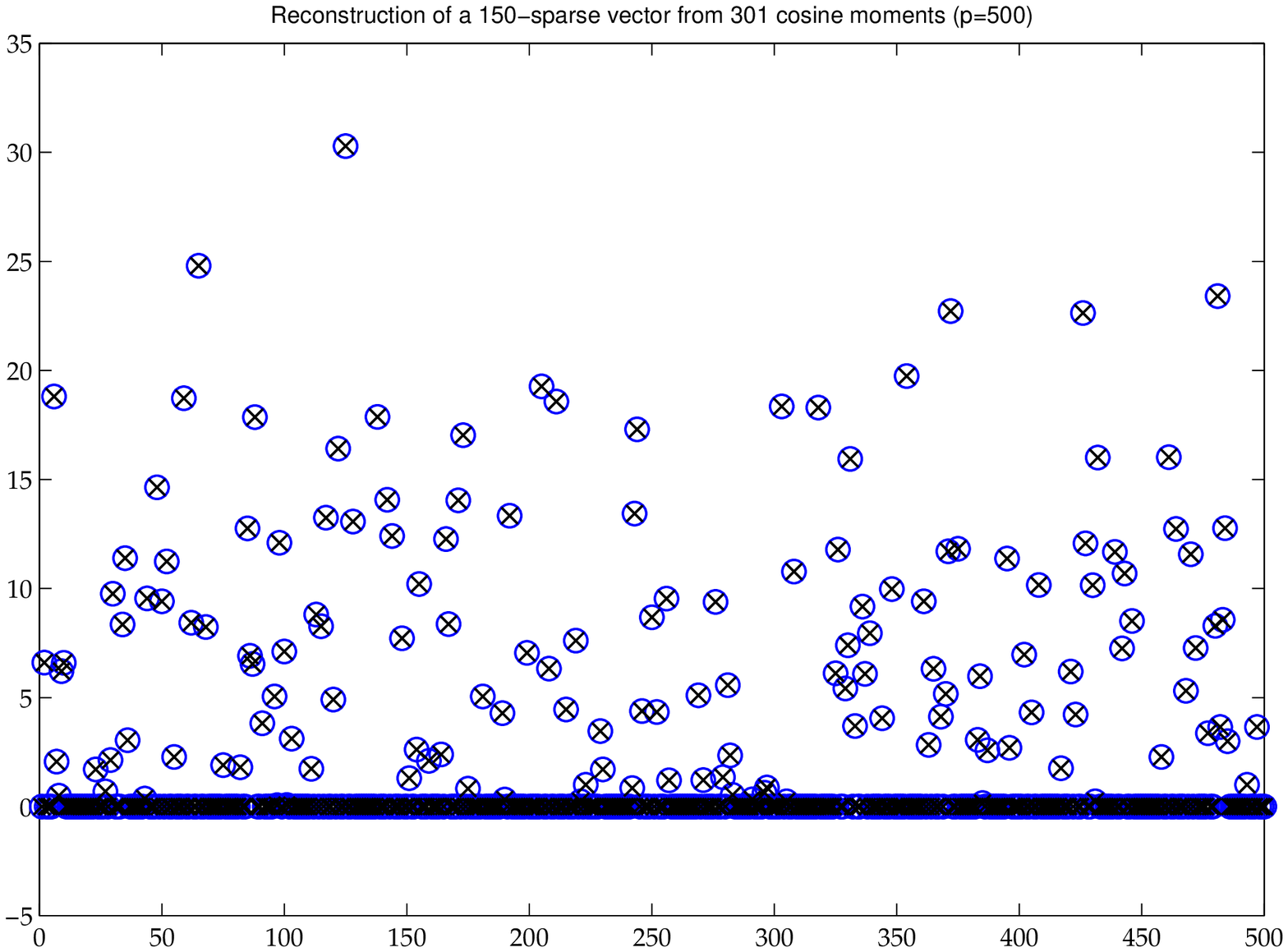}

 \caption{These numerical experiments illustrate Theorem \ref{Theorem
Deterministic}. We consider the family $\mathcal F_{\cos{}}=\left\{1,\cos(\pi x),\cos(2\pi
x),\dotsc\right\}$ and the points $t_k=k/(p+1)$, for $k=1,\dotsc,p$. The blue circles represent the
target vector $x_0$, while the black crosses represent the solution $x^\star$ of
\eqref{BasisPursuit}. The respective values are $s=10$, $n=21$, $p=500$; $s=50$, $n=101$, $p=500$;
and $s=150$, $n=301$, $p=500$.}

 \end{figure}

Note that some coefficients can be badly estimated (for
instance when $s=50$ and $n=101$). This might be due to the fact that we consider the limit case
$n=2s+1$. Nevertheless, this is not the case when we have very few
coefficients ($s=10$ and $n=21$) or a large number of moments ($s=150$ and $n=301$). As a general
rule, we observe faithful reconstruction.
\newpage

\bibliographystyle{amsalpha}
\nocite{*}
\bibliography{dCG_CCS}

\newcommand{\etalchar}[1]{$^{#1}$}
\def\cprime{$'$}
\providecommand{\bysame}{\leavevmode\hbox to3em{\hrulefill}\thinspace}
\providecommand{\MR}{\relax\ifhmode\unskip\space\fi MR }
\providecommand{\MRhref}[2]{%
  \href{http://www.ams.org/mathscinet-getitem?mr=#1}{#2}
}
\providecommand{\href}[2]{#2}
\begin{thebibliography}{CRT06b}

\bibitem[BE95]{MR1367960}
P.~Borwein and T.~Erd{\'e}lyi, \emph{Polynomials and polynomial inequalities},
  Graduate Texts in Mathematics, vol. 161, Springer-Verlag, New York, 1995.
  \MR{1367960 (97e:41001)}

\bibitem[Beu38]{beurling1938integrales}
A.~Beurling, \emph{Sur les int{\'e}grales de fourier absolument convergentes et
  leur application {\`a} une transformation fonctionnelle}, Ninth Scandinavian
  Mathematical Congress, 1938, pp.~345--366.

\bibitem[BEZ94]{MR1299454}
P.~Borwein, T.~Erd{\'e}lyi, and J.~Zhang, \emph{Chebyshev polynomials and
  {M}arkov-{B}ernstein type inequalities for rational spaces}, J. London Math.
  Soc. (2) \textbf{50} (1994), no.~3, 501--519. \MR{1299454 (95j:41015)}

\bibitem[BGI{\etalchar{+}}08]{Ber}
R.~Berinde, A.~Gilbert, P.~Indyk, H.~Karloff, and M.~Strauss, \emph{Combining
  geometry and combinatorics: a unified approach to sparse signal recovery}.

\bibitem[Cal93]{MR1249394}
J.~P. Calvi, \emph{Polynomial interpolation with prescribed analytic
  functionals}, J. Approx. Theory \textbf{75} (1993), no.~2, 136--156.
  \MR{1249394 (94j:41002)}

\bibitem[Car07]{caratheodory1907variabilitetsbereich}
C.~Carath{\'e}odory, \emph{{\"U}ber den variabilit{\"a}tsbereich der
  koeffizienten von potenzreihen, die gegebene werte nicht annehmen},
  Mathematische Annalen \textbf{64} (1907), no.~1, 95--115.

\bibitem[Car11]{caratheodory1911variabilitetsbereich}
\bysame, \emph{{\"U}ber den variabilit{\"a}tsbereich der fourierschen
  konstanten von positiven harmonischen funktionen}, Rendiconti del Circolo
  Matematico di Palermo (1884-1940) \textbf{32} (1911), no.~1, 193--217.

\bibitem[CDD09]{MR2449058}
A.~Cohen, W.~Dahmen, and R.~DeVore, \emph{Compressed sensing and best
  {$k$}-term approximation}, J. Amer. Math. Soc. \textbf{22} (2009), no.~1,
  211--231. \MR{2449058 (2010d:94024)}

\bibitem[CDS98]{MR1639094}
S.~S. Chen, D.~L. Donoho, and M.~A. Saunders, \emph{Atomic decomposition by
  basis pursuit}, SIAM J. Sci. Comput. \textbf{20} (1998), no.~1, 33--61.
  \MR{1639094 (99h:94013)}

\bibitem[CDS01]{MR1854649}
\bysame, \emph{Atomic decomposition by basis pursuit}, SIAM Rev. \textbf{43}
  (2001), no.~1, 129--159, Reprinted from SIAM J. Sci. Comput. {{\bf{2}}0}
  (1998), no. 1, 33--61 (electronic) [ MR1639094 (99h:94013)]. \MR{1854649}

\bibitem[CP10]{RIPless}
E.~J. Cand{\`e}s and Y.~Plan, \emph{A probabilistic and ripless theory of
  compressed sensing}, arXiv (2010).

\bibitem[CRT06a]{MR2236170}
E.~J. Cand{\`e}s, J.~K. Romberg, and T.~Tao, \emph{Robust uncertainty
  principles: exact signal reconstruction from highly incomplete frequency
  information}, IEEE Trans. Inform. Theory \textbf{52} (2006), no.~2, 489--509.
  \MR{2236170 (2007e:94020)}

\bibitem[CRT06b]{MR2230846}
\bysame, \emph{Stable signal recovery from incomplete and inaccurate
  measurements}, Comm. Pure Appl. Math. \textbf{59} (2006), no.~8, 1207--1223.
  \MR{2230846 (2007f:94007)}

\bibitem[Der56]{derry1956convex}
D.~Derry, \emph{Convex hulls of simple space curves}, Canad. J. Math \textbf{8}
  (1956), 383--388.

\bibitem[DG96]{MR1393035}
P.~Doukhan and F.~Gamboa, \emph{Superresolution rates in {P}rokhorov metric},
  Canad. J. Math. \textbf{48} (1996), no.~2, 316--329. \MR{1393035 (97j:44010)}

\bibitem[DJHS92]{donoho1992maximum}
D.~L. Donoho, I.~M. Johnstone, J.~C. Hoch, and A.~S. Stern, \emph{Maximum
  entropy and the nearly black object}, Journal of the Royal Statistical
  Society. Series B (Methodological) (1992), 41--81.

\bibitem[Don92]{donoho1992superresolution}
D.~L. Donoho, \emph{Superresolution via sparsity constraints}, SIAM journal on
  mathematical analysis \textbf{23} (1992), 1309.

\bibitem[Don06]{MR2241189}
\bysame, \emph{Compressed sensing}, IEEE Trans. Inform. Theory \textbf{52}
  (2006), no.~4, 1289--1306. \MR{2241189 (2007e:94013)}

\bibitem[DS89]{MR997928}
D.~L. Donoho and P.~B. Stark, \emph{Uncertainty principles and signal
  recovery}, SIAM J. Appl. Math. \textbf{49} (1989), no.~3, 906--931.
  \MR{997928 (90c:42003)}

\bibitem[DT05]{donoho2005sparse}
D.~L. Donoho and J.~Tanner, \emph{Sparse nonnegative solution of
  underdetermined linear equations by linear programming}, Proceedings of the
  National Academy of Sciences of the United States of America \textbf{102}
  (2005), no.~27, 9446.

\bibitem[DT09]{MR2449053}
\bysame, \emph{Counting faces of randomly projected polytopes when the
  projection radically lowers dimension}, J. Amer. Math. Soc. \textbf{22}
  (2009), no.~1, 1--53. \MR{2449053 (2009k:52015)}

\bibitem[DT10]{donoho2010counting}
\bysame, \emph{Counting the faces of randomly-projected hypercubes and
  orthants, with applications}, Discrete \& Computational Geometry \textbf{43}
  (2010), no.~3, 522--541.

\bibitem[Fel68]{MR0228020}
W.~Feller, \emph{An introduction to probability theory and its applications.
  {V}ol. {I}}, Third edition, John Wiley \& Sons Inc., New York, 1968.
  \MR{0228020 (37 \#3604)}

\bibitem[Fel71]{MR0270403}
\bysame, \emph{An introduction to probability theory and its applications.
  {V}ol. {II}.}, Second edition, John Wiley \& Sons Inc., New York, 1971.
  \MR{0270403 (42 \#5292)}

\bibitem[Fuc96]{fuchs1996linear}
J.-J. Fuchs, \emph{Linear programming in spectral estimation. application to
  array processing}, Acoustics, Speech, and Signal Processing, 1996. ICASSP-96.
  Conference Proceedings., 1996 IEEE International Conference on, vol.~6, IEEE,
  1996, pp.~3161--3164.

\bibitem[Fuc04]{MR2094894}
\bysame, \emph{On sparse representations in arbitrary redundant bases}, IEEE
  Trans. Inform. Theory \textbf{50} (2004), no.~6, 1341--1344. \MR{2094894
  (2005e:94022)}

\bibitem[Fuc05]{fuchs2005sparsity}
\bysame, \emph{Sparsity and uniqueness for some specific underdetermined
  systems}, IEEE International Conference on Acoustics, Speech, and Signal
  Processing, 2005.

\bibitem[GG96]{MR1393430}
F.~Gamboa and E.~Gassiat, \emph{Sets of superresolution and the maximum entropy
  method on the mean}, SIAM J. Math. Anal. \textbf{27} (1996), no.~4,
  1129--1152. \MR{1393430 (97j:44009)}

\bibitem[Hau21a]{MR1544453}
F.~Hausdorff, \emph{Summationsmethoden und {M}omentfolgen. {I}}, Math. Z.
  \textbf{9} (1921), no.~1-2, 74--109. \MR{1544453}

\bibitem[Hau21b]{MR1544467}
\bysame, \emph{Summationsmethoden und {M}omentfolgen. {II}}, Math. Z.
  \textbf{9} (1921), no.~3-4, 280--299. \MR{1544467}

\bibitem[HSS96]{MR1421165}
Z.~He, W.~J. Studden, and D.~Sun, \emph{Optimal designs for rational models},
  Ann. Statist. \textbf{24} (1996), no.~5, 2128--2147. \MR{1421165 (98b:62137)}

\bibitem[IS01]{MR1865340}
L.~A. Imhof and W.~J. Studden, \emph{{$E$}-optimal designs for rational
  models}, Ann. Statist. \textbf{29} (2001), no.~3, 763--783. \MR{1865340
  (2002h:62222)}

\bibitem[KN77]{MR0458081}
M.~G. Kre{\u\i}n and A.~A. Nudel{\cprime}man, \emph{The {M}arkov moment problem
  and extremal problems}, American Mathematical Society, Providence, R.I.,
  1977, Ideas and problems of P. L. \v{C}eby\v{s}ev and A. A. Markov and their
  further development, Translated from the Russian by D. Louvish, Translations
  of Mathematical Monographs, Vol. 50. \MR{MR0458081 (56 \#16284)}

\bibitem[KS53]{MR0059329}
S.~Karlin and L.~S. Shapley, \emph{Geometry of moment spaces}, Mem. Amer. Math.
  Soc. \textbf{1953} (1953), no.~12, 93. \MR{0059329 (15,512c)}

\bibitem[KS66]{MR0204922}
S.~Karlin and W.~J. Studden, \emph{Tchebycheff systems: {W}ith applications in
  analysis and statistics}, Pure and Applied Mathematics, Vol. XV, Interscience
  Publishers John Wiley \& Sons, New York-London-Sydney, 1966. \MR{0204922 (34
  \#4757)}

\bibitem[Rie11]{MR1509135}
F.~Riesz, \emph{Sur certains syst{\`e}mes singuliers d'{\'e}quations
  int{\'e}grales}, Ann. Sci. {\'E}cole Norm. Sup. (3) \textbf{28} (1911),
  33--62. \MR{MR1509135}

\bibitem[Riv90]{MR1060735}
T.~J. Rivlin, \emph{Chebyshev polynomials}, second ed., Pure and Applied
  Mathematics (New York), John Wiley \& Sons Inc., New York, 1990, From
  approximation theory to algebra and number theory. \MR{1060735 (92a:41016)}

\bibitem[Stu88]{Sturmfels:1988:TPM}
B.~Sturmfels, \emph{Totally positive matrices and cyclic polytopes}, Linear
  Algebra and its Applications \textbf{107} (1988), 275--281, Proceedings of
  the Victoria Conference on Combinatorial Matrix Analysis (Victoria, BC,
  1987). \MR{89i:52009}

\end{thebibliography}
\end{document}